\documentclass[preprint,12pt]{elsarticle}
\usepackage[active]{srcltx}
\usepackage{amsmath,bbm}
\usepackage{amssymb}
\usepackage{amsfonts,mathtools,amsthm,thmtools}
\usepackage{MnSymbol}
\usepackage{graphicx}
\usepackage{color}
\usepackage[colorlinks=true,urlcolor=blue,
citecolor=red,linkcolor=blue,linktocpage,pdfpagelabels,bookmarksnumbered,bookmarksopen]{hyperref}

\def\a{\alpha}
\def\b{\beta}
\def\d{\delta}
\def\g{\gamma}
\def\l{\lambda}

\def\n{\nu}
\def\dd{ {\rm d}}
\def\s{\sigma}

\def\e{\varepsilon}

\def\Om{\Omega}

\newcommand{\cA}{{\mathcal A}}
\newcommand{\cB}{{\mathcal B}}

\newcommand{\cF}{{\mathcal F}}
\newcommand{\cL}{{\mathcal L}}
\newcommand{\cK}{{\mathcal K}}
\newcommand{\cJ}{{\mathcal J}}

\newcommand{\cC}{{\mathcal C}}

\newcommand{\fL}{{\mathfrak L}}
\newcommand{\fM}{{\mathfrak M}}

\newcommand{\p}{\partial}

\newcommand{\R}{{\mathbb R}}

\newcommand{\ov}{\overline}

\newcommand{\dom}{\ensuremath{\operatorname{dom}}}

\newcommand{\weakly}{\ensuremath{\rightharpoonup}}

\newcommand{\mres}{\mathbin{\vrule height 1.6ex depth 0pt width
0.13ex\vrule height 0.13ex depth 0pt width 1.3ex}}

\newcommand{\one}{\mathbbm{1}}

\newcommand{\ds}{\displaystyle}

\newcommand{\diver}{{\rm{div}}}
\newcommand{\dist}{{\rm{dist}}}
\newcommand{\iif}{{\rm{if}}}
\newcommand{\iin}{{\rm{in}}}
\newcommand{\oon}{{\rm{on}}}

\newcommand{\otherwise}{{\rm{otherwise}}}
\newcommand{\spt}{{\rm{spt}}}

\newcommand{\Wdqp}{W_\diamond^{1,q'}}
\newcommand{\Wdqpom}{W_\diamond^{1,q'}(\Omega)}

\newcommand{\Wqom}{W^{1,q}(\Omega)}
\newcommand{\Wqstar}{\left(W^{1,q}\right)^*}
\newcommand{\Wqstarom}{\left(W^{1,q}(\Omega)\right)^*}

\newtheorem{remark}{\textbf{Remark}}[section]
\newtheorem{theorem}{\textbf{Theorem}}[section]
\newtheorem{lemma}[theorem]{\textbf{Lemma}}
\newtheorem{corollary}[theorem]{\textbf{Corollary}}
\newtheorem{proposition}[theorem]{\textbf{Proposition}}

\numberwithin{equation}{section}
\newcommand\numberthis{\addtocounter{equation}{1}\tag{\theequation}}

\journal{Journal de Math\'ematiques Pures et Appliqu\'ees}

\begin{document}

\begin{frontmatter}

\title{A variational approach to second order mean field games with density constraints: the stationary case} 

\author{Alp\'ar Rich\'ard M\'esz\'aros}  
\address{Laboratoire de Math\'ematiques d'Orsay, Universit\'e Paris-Sud, 91405 Orsay, France}
\ead{alpar.meszaros@math.u-psud.fr} 

\author{Francisco J. Silva}
\address{Institut de recherche XLIM-DMI, Facult\'e des sciences et techniques Universit\'e de Limoges, 87060 Limoges, France}
\ead{francisco.silva@unilim.fr}

\begin{abstract}
In this paper we study second order stationary Mean Field Game systems under density constraints on a bounded domain $\Omega \subset \R^d$. We show the existence of weak solutions for power-like Hamiltonians with arbitrary order of growth. Our strategy is a variational one, i.e. we obtain the Mean Field Game system as the optimality condition of a convex optimization problem, which has a solution. 
When the Hamiltonian has  a growth of order $q' \in ]1, d/(d-1)[$, the solution of the optimization problem is continuous which implies that the problem  constraints are qualified. Using this fact and the computation of the subdifferential of a convex functional introduced by Benamou-Brenier (see \cite{BB}), we prove the existence of a solution of the MFG system. In the case where the Hamiltonian has a growth of order $q'\geq d/(d-1)$, the previous arguments do not apply and we prove the existence by means of an approximation argument. 

\vspace{1cm}
\noindent{\bf R\'esum\'e}
\vspace{0.5cm}

\noindent Dans ce papier on \'etudie des syst\`emes de jeux \`a champ moyen sous contrainte de densit\'e sur un domaine born\'e $\Omega \subset \R^d$. On d\'emontre l'existence de solutions faibles pour des hamiltoniens de type puissance avec ordre de croissance arbitraire. Notre strat\'egie est variationnelle, on obtient le syst\`eme de jeux \`a champ moyen comme condition d'optimalit\'e d'un probl\`eme convexe, lequel a une solution. Quand l'hamiltonien a un ordre de croissance $q' \in  {]}1, d/(d-1)[$, la solution du probl\`eme d'optimisation est continue, ce qui implique que les contraintes du probl\`eme sont qualifi\'ees. En utilisant cette propriet\'e et le calcul du sous-diff\'erentiel d'une fonctionnelle convexe introduite par Benamou-Brenier (voir \cite{BB}), on d\'emontre l'existence d'une solution du syst\`eme MFG. Dans les cas o\`u l'hamiltonien a un ordre de croissance $q'\geq d/(d-1)$, les arguments pr\'ec\'edents ne sont pas applicables et on montre l'existence avec un argument d'approximation. 

\end{abstract}


\end{frontmatter}

\section{Introduction}
The theory of Mean Field Games (shortly MFG in the sequel) was introduced recently and simultaneously  by J.-M. Lasry and P.-L. Lions (\cite{LasLio06i,LasLio06ii,LasLio07}) and  M. Huang, R. P. Malham{\'e} and P. E. Caines (see \cite{HuaMalCai}).
 The main objective  of the MFG theory is the study of the limit behavior of Nash equilibria for symmetric differential games with a very large number of ``small'' players. In its simplest form, as the number of players tends to infinity, limits of Nash equilibria can be characterized in terms of the solution of the following coupled PDE system:
%
%
\begin{equation}
\left\{
\begin{array}{rcll}
-\partial_t u(t,x) -\n\Delta u(t,x) + H(x,\nabla u(t,x)) &=& f[m(t)](x)& {\rm{in\ }}(0,T]\times\R^d\\[5pt]
\partial_t m(t,x)-\n\Delta m(t,x) - \diver\left(\nabla_p H(x,\nabla u(t,x))m(t,x)\right) &=& 0 & {\rm{in\ }}(0,T]\times\R^d,\\[5pt]
m(0,x)=m_0, \;\;\; u(T,x) = g(x) & & & {\rm{in\ }}\R^d,
\end{array}
\right. \tag{MFG}\label{MFG}
\end{equation}
%
where $H(x,\cdot)$ is convex. The Hamilton-Jacobi-Bellman (HJB) equation in \eqref{MFG} characterizes the value function $u[m]$ associated to a stochastic optimal control problem solved by a typical player whose cost function depends at each time $t$ on the distribution $m(t,\cdot)$ of the other agents. We remark that this interaction  can be global, e.g. if  $f[m(t,\cdot)](x)$ is a convolution of $m(t,\cdot)$ with another function, or local, i.e. when $f[m(t)](x)$ can be identified to a function $f(x,m(t,x))$. 
 The Fokker-Planck equation (FP) in \eqref{MFG} describes the evolution $m[u]$ of the initial distribution $m_0$ when all the agents follow the optimal feedback strategy computed by the typical agent. We refer the reader to the original papers \cite{LasLio06i,LasLio06ii,LasLio07} and the lectures \cite{curslions} for more details on the modeling  and the relation with the system \eqref{MFG}. See also  \cite{Car_not,GomSau14} for a survey on the subject. 
 
For local couplings $f(\cdot,m)$, system \eqref{MFG} can be obtained (at least formally) as the optimality condition of   problem
\begin{equation}\label{var}
\begin{array}{l}
\min \ds \int_0^T\int_{\R^d} \left\{ m(t,x)L\left(x,-\frac{w(t,x)}{m(t,x)}\right)+F(x,m(t,x))\right\}\,\dd x\dd t+\int_{\R^d}g(x)m(T,x)\,\dd x,\\[6pt]
\mbox{s.t.} \hspace{0.3cm}\partial_t m -\nu \Delta m+\diver(w)=0,\;\;\; m(0,x)=m_0,
\end{array}
\end{equation} 
with $\ds F(x,m):= \int_{0}^{m}f(x,m')\,\dd m'$, $L(x,v):= H^{*}(x,v)$ (where the Fenchel conjugate  $H^{*}(x,v)$ is calculated on the second variable of $H$)  and $m_0\in L^{\infty}(\R^d)$ satisfying that  $m_0\geq 0$ and $\ds\int_{\R^d} m_0\, \dd x=1$.  This type of approach, including  also the degenerate first order case ($\nu=0$),   has been  studied extensively in the last years in a series of papers \cite{Car13, Gra14, CarGra, CarGraPorTon}. The optimization problem above recalls the so-called Benamou-Brenier formulation of the 2-Wasserstein distance between two probability measures, which gives a fluid mechanical or dynamical interpretation of the Monge-Kantorovich optimal transportation problem (see \cite{BB,CarCarNaz}). We refer the reader to \cite{AchCamCap12}, \cite{LacSolTur} and the recent work \cite{BenCar} for some optimization methods to solve numerically \eqref{MFG} based on the formulation \eqref{var}.

With a well-chosen time-averaging procedure, one can introduce stationary MFG systems  as an ergodic limit of time dependent ones (see \cite{CarLasLioPor, CarLasLioPor2}),
\begin{equation}
\left\{
\begin{array}{rcll}
-\n\Delta u(x) + H(x,\nabla u(x)) -\l&=& f(x,m(x))& {\rm{in\ }}\R^d\\[5pt]
-\n\Delta m(x) - \diver\left(\nabla_p H(x,\nabla u(x))m(x)\right) &=& 0 & {\rm{in\ }}\R^d,\\[5pt]
\ds\int_{\R^d}m(x)\,\dd x=1, \;\;\; \int_{\R^d}u(x)\, \dd x= 0, \; \; \; m\geq 0. & & &
\end{array}
\right. \tag{$MFG_{\infty}$}\label{MFGinf}
\end{equation}
At least formally\eqref{MFGinf} can be obtained  as the first order optimality condition of the problem
\begin{equation}\label{var2}
\begin{array}{l}
\min\ds \int_{\R^d} \left\{ m(x)L\left(x,-\frac{w(x)}{m(x)}\right)+F(x,m(x))\right\}\,\dd x,\\[4pt]
\mbox{s.t. } \hspace{0.3cm} -\Delta m+\diver(w)=0,\;\;\; \ds\int_{\R^d}m(x)\,\dd x=1,\;\;\; m\geq 0.
\end{array}
\end{equation}
The existence of smooth solutions  for both evolutive and stationary MFG systems has been obtained in various settings in a series of papers  (see for instance \cite{GomPimSan15, GomPirSan12, GomMit14, GomPatVos14} and the references therein). The used techniques combine variational arguments and sharp PDE estimates. Connections between stationary MFG systems and the so-called Evans-Aronsson problem have been also recently studied in \cite{GomSan14}.

The objective of this work is to rigorously study the optimization problem \eqref{var2} with the additional constraint  $m\le 1$ a.e. Formally this should be linked to a system like\eqref{MFGinf} with {\it $m\le 1$ a.e.} and an additional Lagrange multiplier corresponding to the new constraint.  Moreover, in view of the interpretation of \eqref{MFG} as a continuous Nash equilibria, we expect that our derivation of an MFG system with a density constraint is  linked to symmetric games with a large number of players on which  ``hard congestion'' constraints are imposed. Similar models in the framework of   crowd motion, tumor growth, etc. have been  already studied in the literature (see  for instance \cite{MauRouSan1,MauRouSan2}). In the case of MFG systems, we refer the reader to the papers \cite{BurDiFMarWol} (for evolutive systems) and \cite{GomMit14} (for stationary systems), in which   ``soft-congestion'' effects, meaning that people slow down when they arrive to congested zones, are studied. Let us remark  that in \cite{curslions} it is also explained how to study systems like \eqref{MFGinf} by means of a (degenerate) elliptic equation in space-time. However, this approach with the additional constraint  $m\le 1$ a.e. seems to be ineffective. 

The question of hard congestion effects/density constraints for MFG systems was first raised in \cite{filippo}. More precisely,  in the cited reference the author asks  if a  MFG system can be obtained with the additional constraint that the density of the population does not exceed a given threshold, for instance 1. To the best of our knowledge, this work is the first attempt to investigate this question. The stationary setting plays an important role in our study and we expect   to extend our results to the dynamic case in some future research. 

Let $\Om\subset\R^d$ ($d\ge2$) be a non-empty bounded open set with smooth boundary and such that the Lebesgue measure of $\Omega$ is strictly greater than $1$.   Moreover, let $f:\Om\times\R\to\R$ be a continuous function which is non-decreasing in the second variable and define $\ell_q:\mathbb{R} \times\mathbb{R}^d\to\overline{\R}$ (with $\overline{\R}:= \R\cup \{+\infty\}$) and $\mathcal{L}_q:W^{1,q}(\Omega)\times L^q(\Omega)^d\to\overline{\R}$   as
\begin{equation}\label{definicionfuncionesdebase}\ell_q(a,b):=\left\{
\begin{array}{ll}
\frac{1}{q}\frac{|b|^q}{a^{q-1}}, & \iif\ a>0,\\
0, & \iif\ (a,b)=(0,0),\\
+\infty, & \otherwise.
\end{array}
\right. \hspace{0.3cm} \mathcal{L}_q(m,w):=\int_\Omega \ell_q(m(x),w(x))\, \dd x.\end{equation}
We consider the problem 
\begin{equation}
\begin{array}{l}
\ds\min \hspace{0.4cm} \mathcal{L}_q(m,w)+\int_{\Omega}F(x,m(x))\,\dd x,\\[8pt]
\mbox{s.t. } \hspace{0.3cm} -\Delta m+\diver(w)=0 \; \; \mbox{in $\Omega$}, \; \; \; (\nabla m - w)\cdot n =0 \hspace{0.4cm}  \oon\ \partial\Omega, \\[4pt]
\hspace{2cm} \ds\int_{\Omega}m(x)\,\dd x=1,\;\;\; 0\leq m \leq 1,
\end{array} \tag{$P_q$}\label{Pqin}
\end{equation}
where, as before, $F(x,m)$ is an antiderivative  of $f(x,m)$ with respect to the second variable. We divide our main results in two classes, depending on the value of $q$.\smallskip

{\it Case 1: $q>d$}.  In this case, using the classical direct method of the calculus of variations, we prove the existence of a solution $(m,w)$ of  \eqref{Pqin}. Using that $m\in W^{1,q}(\Omega)\hookrightarrow C(\overline{\Omega})$, we are able to compute the subdifferential of $\mathcal{L}_r(m,w)$ for any $1<r\leq q$. It seems that this type of result is new in the literature. Moreover, the continuity of $m$ allows us to prove that the constraints in  \eqref{Pqin} are qualified (see e.g. \cite[Chapter 2]{BonSha}). Using the computation of the subdifferential with $r=q$ and classical arguments in convex analysis, we derive the existence of $u \in W^{1,s}(\Omega)$ $(s\in [1, d/(d-1)[)$, $\lambda \in \R$ and two nonnegative regular measures $\mu$ and $p$   such that 
\begin{equation}
\left\{
\begin{array}{ll}
-\Delta u + \frac{1}{q'}|\nabla u|^{q'}+\mu - p  - \l=f(x,m), & \iin\ \Om,\\[6pt]
-\Delta m - \diver\left(m |\nabla u|^\frac{2-q}{q-1}\nabla u\right) = 0, & \iin\ \Omega, \\[6pt]
\nabla m\cdot n= \nabla u\cdot n =0, & \oon\ \partial\Omega,\\[6pt]
\ds\int_\Omega m\,\dd x=1,\ \ 0\le m\le 1 , & \iin\ \Omega,\\[8pt]
\mbox{{\rm spt}}(\mu) \subseteq \{m=0\},   \hspace{0.3cm}  \mbox{{\rm spt}}(p) \subseteq \{m=1\}, & \;
\end{array}
\right.\tag{$MFG_{q}$}\label{MFGqin}
\end{equation}
where the system of PDEs is satisfied in the weak sense, `$\mbox{{\rm spt}}$' denotes the support of a measure and $q':=
q/(q-1)$. In the above system, $p$ appears as a Lagrange multiplier associated to the constraint $m\leq 1$ and can be interpreted as a sort of  a ``pressure'' term. We also compute the dual problem associated to   \eqref{Pqin}  recovering \eqref{MFGqin} by duality.  Finally, in the open set $\{ 0 < m< 1\}$ we prove some local regularity results for the pair $(m,u)$.  
%

{\it Case 2: $1<q\leq d$}. In this case, even if the existence of a solution still holds true, $m$ is in general discontinuous, which implies that the arguments employed in the  computation of the subdifferential of  $\mathcal{L}_q(m,w)$ are no longer valid. Moreover,  the discontinuity of $m$ implies that the constraint $0\leq m\leq 1$ is in general not qualified. In order to overcome these issues, we use an approximation argument. By  adding the term $\e \mathcal{L}_r(m,w)$ with $r>d$ to the cost function and using the arguments in {\it Case 1} we obtain a system similar to \eqref{MFGqin} depending on $\e$. Then, by means of some uniform bounds with respect to $\e$ and recent results on estimates on the gradients for solutions of elliptic equations with measure data (see \cite{mingione2}), as $\e \downarrow 0$ we can prove the existence of limit points satisfying \eqref{MFGqin} where the concentration properties for $p$ and $\mu$ have to be understood in a weak sense. \smallskip

The structure of the paper is as follows: in Section \ref{preliminary} we set the basic notations and prove some preliminary results including the computation of the subdifferential of $\mathcal{L}_q(m,w)$. In Section \ref{problemstatement} we define rigorously problem \eqref{Pqin} for the case $q>d$ and we prove the existence of a solution as well as the qualification property of the constraints. 
In Section \ref{optimalityconditions} we characterize the solutions of  \eqref{Pqin} in terms of \eqref{MFGqin} still in the case $q>d$. Moreover, we prove some local regularity results and we derive the dual problem.  The uniqueness of the solutions is also discussed.  In Section \ref{aproximationargument} we complete the proof of the previous statements for any $1<q\le d$ by means of an approximation argument. Finally, in the Appendix we recall some important results about elliptic equations with irregular right hand sides. 
\vspace{10pt}

{\bf Acknowledgement:} Both authors would like to thank  F. Santambrogio for his constant interest and  valuable discussions during the preparation of this paper. They are particularly thankful for his suggestion of the approximation procedure to treat the less regular cases. The first author would like to thank the hospitality of the University of Limoges several times during this project.  He also thanks the hospitality of the Fields Institute, Toronto, in the framework of the {\it Special Semester on Variational Problems in Economics, Physics and Geometry}, Fall 2014.  Both authors thank the partial support of the project iCODE: {\it``Large-scale systems and Smart grids: distributed decision making''} -- {\it ``strategic crowds''}. F. J. Silva  benefited also from the support of the {\it ``FMJH Program Gaspard Monge in optimization and operation research''}, and from the support to this program from EDF.
Last but not least, we thank L. Brasco for the discussions on the Calder\'on-Zygmund theory, directing us towards the references on this topic used in this paper.

\section{Notations and Preliminary Results}\label{preliminary}

We  first fix some standard notation. Let $\Omega \subseteq \R^{d}$ ($d\ge 2$) be a non-empty, bounded open set with a smooth boundary, satisfying a uniform interior ball condition,  and denote by  $n$ the outward normal to $\p\Omega$.  Let us set  $|\cdot|$ for the usual euclidean norm on $\R^d$ and, given a Lebesgue measurable set $A\subseteq \R^d$, if it is not ambiguous,  we also use  $|A|$ for its $d$-dimensional Lebesgue measure.

We denote by $\fM(\ov{\Omega})$  the space of (signed) Radon measures defined on ${\ov\Omega}$. We set $\fM_+(\ov{\Omega})$ (respectively $\fM_-(\ov{\Omega})$) for the subset of   $\fM(\ov{\Omega})$ of   non-negative (respectively non-positive) Radon measures. Given the Hahn-Jordan decomposition $m=m^+- m^-$, with $m^+$, $m^- \in \fM_+(\ov{\Om})$, we set $|m|_{TV}:=m^+(\ov{\Om})+m^-(\ov{\Om})$ for the total variation of $m$.  We also denote by $\fM_{ac}(\ov{\Omega})$ and $\fM_s(\ov{\Omega})$ the spaces of absolutely continuous and singular measures w.r.t. the Lebesgue measure, respectively.  For notational convenience, if  $m\in \fM_{ac}(\ov{\Omega})$ we will also denote by $m$ its density w.r.t. the Lebesgue measure. Given   $\mu \in \fM(\ov{\Omega})$ we set   $\mu\mres A$ for its restriction to  $A\subseteq\Om$, defined as $\mu\mres A(B):=\mu(A\cap B)$ for all $B\in\cB(\Om)$ (where   $\cB(\Om)$ denotes the Borel $\sigma$-algebra on $\Om$). Finally, given $A\subseteq \R^d$, we set $\chi_A(x)=0$ if $x\in A$ and $+\infty$ otherwise. Moreover we set $\one_A(x)=1$ if $x\in A$ and $0$ otherwise.
 
Now, let $q>1$ be given and set  $q':= q/(q-1)$. Consider  the sets 
$$\begin{array}{c}A_{q'}:=\{(a,b)\in\mathbb{R}\times\mathbb{R}^d:a+\frac{1}{q'}|b|^{q'}\le 0\},\\[4pt]
\mathcal{A}_{q'}:=  \left\{(a, b) \in L^{\infty}(\Omega)\times L^{\infty}(\Omega)^{d}, \; \; (a(x), b(x)) \in A_{q'}, \; \; \mbox{{\rm for a.e.}  $x \in \Omega$}\right\},\end{array}$$
and recall the functions $\ell_q$ and  $\mathcal{L}_q$ defined in \eqref{definicionfuncionesdebase}. We have the following result.
\begin{lemma}\label{subdiff_cost} 
Suppose that $q>d$ and  let $ 1<r\leq q$. Then, the following assertions hold true:\smallskip\\
{\rm(i)} The closure of $\mathcal{A}_{r'}$ in $\Wqstarom \times  L^{q'}(\Omega)^d $ is given by 
\begin{equation}\label{expressionovA}\overline{\mathcal{A}_{r'}}= \left\{ (\alpha, \beta) \in \fM(\ov{\Omega}) \times L^{r'}(\Om)^d\; : \; \alpha + \frac{1}{r'} |\beta|^{r'} \leq 0 \right\}, \end{equation}
where the inequality in \eqref{expressionovA} means that for every non-negative $\phi \in C(\overline{\Omega})$ we have that 
\begin{equation}\label{inequalitymeasures} \int_{\ov{\Omega}} \phi(x)\, \dd \alpha(x)+  \frac{1}{r'}  \int_{\Omega} \phi(x)  |\beta(x)|^{r'}\, \dd x \leq 0.\end{equation}
{\rm(ii)}  Restricted to $W^{1,q}(\Omega) \times L^{q}(\Omega)^d$,  the functional $\cL_r$ is convex and l.s.c. Moreover,  for every $(m,w)\in W^{1,q}(\Omega)\times L^q(\Omega)^d,$   it holds that
\begin{equation}\label{expressionsforL}
\begin{array}{ll}
\ds\mathcal{L}_r(m,w)&= \ds\sup_{(\alpha,\b)\in \mathcal{A}_{r'}}   \int_{\Omega} \left[ \a(x) m(x) + \b(x)\cdot w(x) \right]  \,\dd x\\[6pt]
&=  \ds\sup_{(\alpha,\b)\in \overline{\mathcal{A}_{r'}}} \left[\int_{\ov{\Omega}} m(x)\, \dd \alpha(x)+ \int_{\Omega} \beta(x)\cdot w(x)\,  \dd x\right],
\end{array}
\end{equation}
and $\mathcal{L}_r^{*}(\alpha, \beta)= \chi_{\ov{\cA_{r'}}}(\alpha, \beta)$ for all $(\alpha,\beta)\in \Wqstarom \times L^{q'}(\Omega)^{d}$.
\end{lemma}
\begin{proof}[Proof of $\rm{(i)}$]
 Let $(\alpha_{n}, \beta_{n})  \in \mathcal{A}_{r'}$ be a sequence converging to some  $(\alpha, \beta)$ in $\Wqstarom \times  L^{q'}(\Omega)^d$. Then, for any non-negative $\phi \in W^{1,q}(\Omega)$ we have that 
$$ \int_{\Omega} \phi(x) \alpha_{n}(x)\,\dd x \leq -\frac{1}{r'} \int_{\Omega}  \phi(x)  |\beta_{n}(x)|^{r'}\, \dd x.$$
Since $\b_n\to \b$ in $L^{q'}(\Om)^d,$ except for some subsequence, $|\b_n(x)|^{r'}\to|\b(x)|^{r'}$ for a.e. $x\in\Om.$ Having positive integrands in the second integral, by Fatou's lemma we obtain
$$ \langle\a,\phi\rangle_{\Wqstar,W^{1,q}}\leq  -\frac{1}{r'} \int_{\Omega}  \phi(x)  |\beta(x)|^{r'} \dd x \hspace{0.3cm} \mbox{for all $\phi \in W^{1,q}(\Omega),\ \phi\ge0$}.$$
In particular, letting $\phi \equiv 1$, we have that  $\b\in L^{r'}(\Om)^d$ and by \cite[Chapitre I, Th\'eor\`eme V]{Schw66} we can extend $\alpha$ to a linear functional over $C(\overline{\Omega})$, i.e. to an element in $ \fM(\ov{\Omega})$, satisfying \eqref{inequalitymeasures}. This proves one inclusion in \eqref{expressionovA}. 

In order to prove the converse inclusion, let  $(\alpha, \beta)$ be an element of the r.h.s. of  \eqref{expressionovA}. Equivalently, 
$$\a^{\rm{ac}}+\frac{1}{r'}|\b|^{r'}\le 0\;\; {\rm{a.e.\ in\ }}\Om\;\;\; {\rm{and}}\ \a^{\rm{s}}\le 0,$$
where $\a^{\rm{ac}}$ and $\a^{\rm{s}}$ denote the absolutely continuous and singular parts of $\a$ with respect to the Lebesgue measure, respectively. We shall construct different approximations for $\a^{\rm{ac}}$ and $\b$ on the one hand and for $\a^{\rm{s}}$ on the other hand.  For $\gamma>0$ and $x\in \R^d$ we set $B_{\gamma}(x)=\{ y\in \R^d \; ; \; |y-x|<\gamma\}$. Consider a mollifier $\eta: \R^d \to \R$ satisfying that $\eta \in C_c^{\infty}(\R^d)$, $\eta\geq 0$, $\ds\int_{\R^d}\eta(x)\,\dd x=1,$ $\spt(\eta)\subseteq B_1(0)$ and $\eta(x)=\eta(-x)$ for all $x\in \R^d$. Now, for $\e>0$ set
$$\Omega_\e:=\{x\in\Omega:\ \dist(x,\partial\Omega)>\e\}, \hspace{0.3cm} \eta_{\e}(x):= \frac{1}{\e^{d}} \eta(x/\e),$$ and for all $x\in \Omega$ and $i=1,..., d$, let us define
$$\tilde\alpha_{\e}(x):= \int_{\Omega} \eta_{\e}(x-y)  \alpha^{ac}(y)\dd y \one_{\Om_\e}(x) , \hspace{0.5cm} \tilde\b_{\e}^{i}(x):=\int_{\Omega} \eta_{\e}(x-y)\beta^{i}(y) \dd y \one_{\Om_\e}(x). $$

By convexity and Jensen's inequality, for all $x \in \Omega_\e$ we have that
\begin{equation}\label{cut}
 \tilde\a_\e(x)+\frac{1}{r'}|\tilde\b_\e(x)|^{r'}\le \left(\a+\frac{1}{r'}|\b|^{r'}\right) \ast \eta_\e (x)\leq 0,
\end{equation}
and so $( \tilde\a_\e,  \tilde\b_\e)\in \mathcal{A}_{r'}$ and one easily checks that $\tilde\a_\e \to \a^{ac}$ in $\fM(\overline{\Omega})$ and $ \tilde\b_\e \to \b$ in $L^{q'}(\Omega)^{d}$. 

In order to approximate $\a^{\rm{s}}$ let us define the following kernel: for $x\in\ov\Om$ and $\e>0$ let us set $\ds\rho^x_\e:= \left(\one_{B_\e(x)\cap\ov\Om}\right) / |B_\e(x)\cap\ov\Om|$. Note that for all $x\in\ov\Om$ we have that $\rho^x_\e \to \d_x$  in $\fM(\overline{\Omega})$ as $\e\downarrow 0$. Given $y\in\ov\Om$ and $\e>0$ let us define 
$$\hat{\a}_\e(y):=\int_{\ov\Om}\rho_\e^x(y)\,\dd\a^{\rm{s}}(x).$$
Observe that for all $\e>0$ the function $\hat{\a}_\e$ is non-positive and, due to our regularity assumption on $\partial \Omega$, we have that  $\hat{\a}_\e\in L^\infty(\Om)$. 
Let us show that $\hat{\a}_\e \to \a^{\rm{s}}$  in $\fM(\overline{\Omega})$. For any $\phi\in C(\ov\Om)$, Fubini's theorem yields
\begin{align*}
\int_{ \Om}\phi(y)\hat{\a}_\e(y)\,\dd y&=\int_{ \Om}\phi(y)\int_{\ov\Om}\rho_\e^x(y)\,\dd\a^{\rm{s}}(x)\,\dd y=\int_{\ov\Om}\int_{ \Om}\phi(y)\rho_\e^x(y)\,\dd y\,\dd\a^{\rm{s}}(x)\\[5pt]
&\to\int_{\ov\Om}\phi(x)\,\dd\a^{\rm{s}}(x)\;\; {\rm{as}}\ \e\downarrow 0,
\end{align*}
where we have  used  that $\phi$ is uniformly continuous in $\ov{\Omega}$ (since this set is compact) and so 
$$\int_{ \Om}\phi(y)\rho_\e^x(y)\,\dd y\to\phi(x)\;\;\; \mbox{\rm{uniformly} in $\ov{\Om}$  as $\e\downarrow 0$.}$$
This proves the convergence of $\hat{\a}_\e$. Defining, $\alpha_\e:= \hat{\a}_\e+\tilde\a_\e$ we have that $(\alpha_\e, \tilde\b_\e) \in  \cA_{r'}$ and  $(\a_\e,\tilde \b_\e)\to(\a,\b)$ in $\fM(\overline{\Omega})\times L^{q'}(\Omega)^{d}$.   The embedding $\fM(\overline{\Omega}) \hookrightarrow \Wqstarom$ implies that    the convergence also holds in  $(W^{1,q}(\Omega))^{*} \times L^{q'}(\Omega)^{d}$,  from which   assertion {\rm(i)} follows. \smallskip\\
%
%
%
{\it Proof of {\rm (ii)}.} It suffices to show \eqref{expressionsforL} (here we remark that by the Sobolev embedding we identify $m$  with an element in  $C(\ov\Om),$ hence the second integral is   meaningful). Indeed, \eqref{expressionsforL}  shows that   $\cL_r$ is the supremum of linear and continuous functionals, hence it is convex and l.s.c. 
 For $k \in \mathbb{N}$ set $ A_{r',k}:= \left\{ (a,b) \in A_{r'} \; ; \;   a  \geq -k, \; \max_{i=1,\hdots, d}|b_i|\leq k\right\}$.  Since a.e. in  $ \Omega$ we have that 
$$ \lim_{k \to \infty} \sup_{(a,b)\in A_{r',k}} \left\{ a m(x) + b\cdot w(x)\right\}= \sup_{(a,b)\in  A_{r'}}  \left\{ a m(x) + b\cdot w(x)\right\}, $$
and $(a,b)=(0,0) \in  A_{r',k}$, by monotone convergence we have that 
\begin{equation}\label{ec:1} \mathcal{L}_r(m,w)=  \lim_{k \to \infty} \int_{\Omega} \sup_{ (a,b)\in  A_{r',k}} [ a m(x) + b\cdot w(x)]\,\dd x.\end{equation}
Note that if $\left| \{m<0\}\right|>0$, then by  \eqref{expressionovA}, we readily check that both sides in    \eqref{expressionsforL} are equal to $+\infty$.
On the other hand, note that for every $(m,w)\in\R\times\R^d$ with $m\geq 0$ there exists a unique pair $(a(m,w), b(m,w))\in\R\times\R^d$ such that 
$$ \sup_{ (a,b)\in  A_{r',k}} \left\{ a m  + b\cdot w\right\}=  a(m,w)m  + b(m,w)\cdot w. $$
Indeed,
\begin{equation}\label{ec:argmax} b(m,w)= \mbox{argmax}_{|b|_{\infty}\leq k} \left\{  -\frac{|b|^{r'} }{r'} m  + b\cdot w\right\} \; \;  \; \; \mbox{and } \;  a(m,w)=-\frac{|b(m,w)|^{r'} }{r'},\end{equation}
which are well-defined by the strict concavity of the objective function. Moreover
this implies that  $ \R_{+} \times \R^{d}\ni (m,w) \mapsto (a(m,w), b(m,w))\in \R \times \R^{d}$ is continuous and measurable and thus $W^{1,q}(\Omega) \times (L^{q}(\Omega))^{d}\ni (m,w) \mapsto (a(m,w), b(m,w)) \in L^{\infty}(\Omega) \times L^{\infty}(\Omega)^{d}$ is well defined.  Therefore, defining  $  \mathcal{A}_{r',k}:=\left\{(a,b)\in L^{\infty}(\Omega) \times L^{\infty}(\Omega)^{d}, \; \;  (a(x),b(x)) \in  A_{r',k} \; \; \mbox{for a.e. $x\in\Omega$} \right\}$ we get that 
$$\int_{\Omega} \sup_{ (a,b)\in  A_{r',k}} [ a m(x) + b\cdot w(x)]\,\dd x= \sup_{(a,b)\in  \mathcal{A}_{r',k}} \int_{\Omega}  [a(x) m(x) + b(x)\cdot w(x)]\, \dd x,$$
which, together with  \eqref{ec:1}, implies  that 
\begin{align*} 
\mathcal{L}_r(m,w) &=  \lim_{k \to \infty} \sup_{(a,b)\in  \cA_{r',k}} \int_{\Omega}  [a(x) m(x) + b(x)\cdot w(x)]\, \dd x\\[6pt]
&=\sup_{(a,b)\in \mathcal{A}_{r'}} \int_{\Omega}  [a(x) m(x) + b(x)\cdot w(x)]\, \dd x,
\end{align*}
proving the first equality in \eqref{expressionsforL}. The second equality follows from {\rm(i)} and the continuity of the considered linear application. Finally, the identity $\cL^*_r=\chi_{\ov\cA_{r'}}$ is a consequence of ${\rm{(i)}}$ and \eqref{expressionsforL}.
\end{proof}
\begin{remark}
We refer the reader to \cite[Chapter 5]{filippobook15} for the proof of the semicontinuity of $\cL_r$ in a more general setting.
\end{remark}

For  $m\in W^{1,q}(\Omega)$  denote $E_0^m:=\left\{x\in\ov\Om:m(x)=0\right\}$ and $E_{1}^m=\left\{x\in\ov\Om:m(x)>0\right\}$. Note that if $q>d$, then $E_0^{m}$ is closed. 

\begin{theorem}\label{subdifferentialcomputation}  Let $(m,w)\in W^{1,q}(\Omega)\times L^{q}(\Omega)^d$ ($q>d$) and $1<  r \leq q$.  Suppose that $\cL_r(m,w)<\infty$.     Then,  if  $v:= (w/m)\one_{E_1^m}\notin L^{r}(\Omega)^d$ we have that $\partial\cL_r(m,w)=\emptyset$. 
Otherwise,    $\cL_r$ is subdifferentiable at $(m,w)$ and  
\begin{equation}\label{subdifferential}
\partial\mathcal{L}_r(m,w)=\left\{(\a,\b)\in \overline{\cA_{r'}} \; : \; \a\mres E_1^m=-\frac{1}{r'}|v|^r \; \; \; {\rm and } \; \;  \;  \b\mres E_1^m=|v|^{r-2}v\right\}.
\end{equation} \normalsize
In particular, the singular part of $\alpha$ is concentrated in $E_0^m$. 
\end{theorem}

\begin{proof} First note that since $\mathcal{L}_r(m,w)<\infty$, we have that  $\left| \{m<0\}\right|=0$ and  $w=0$ a.e. in $E_0^m$. By Lemma \ref{subdiff_cost}  for all $(m,w)\in W^{1,q}(\Omega)\times L^q(\Omega)^d,$ $m\ge 0$ we have that 
$$ \partial\mathcal{L}_r(m,w)= \mbox{argmax}_{(\alpha,\beta) \in \overline{\mathcal{A}_{r'}}} \left\{\int_{\ov{\Omega}}  m\,  \dd \a + \int_{\Omega} \b\cdot w\,  \dd x\right\}. $$
We claim that
\begin{equation}\label{mewnrenwrbw}\sup_{(\alpha,\beta)\in \ov{\cA_{r'}}} \left\{\int_{\ov{\Omega}} m \,  \dd \alpha  + \int_{\Omega}\beta \cdot w\, \dd x\right\}=\sup_{\b\in    L^{r'}(\Om)^d}-\frac{1}{r'}\int_\Omega m|\beta|^{r'}\,\dd x+\int_\Omega \beta \cdot  w\,\dd x.
\end{equation}
Indeed, the inequality ``$\ge$'' is immediate. To show the converse inequality  for every $\e>0$  let $(\a_\e,\b_\e) \in  \ov{\cA_{r'}}$ such that 
$$\int_{\ov{\Omega}} m(x)\,  \dd \alpha_\e(x) + \int_{\Omega}\beta_\e(x)\cdot w(x)\, \dd x \geq \sup_{(\alpha,\beta)\in \ov{\cA_{r'}}} \left\{\int_{\ov{\Omega}} m(x)\,  \dd \alpha(x) + \int_{\Omega}\beta(x)\cdot w(x)\, \dd x\right\}-\e.$$
Then, denoting by $ \hat{s}$ the r.h.s. of  \eqref{mewnrenwrbw}, by  \eqref{expressionovA} and the previous inequality we have that 
$$ \hat{s} \geq -\frac{1}{r'}\int_\Omega m|\beta_\e|^{r'}\,\dd x+\int_\Omega \beta_{\e}\cdot  w\,\dd x \geq \sup_{(\alpha,\beta)\in \ov{\cA_{r'}}} \left\{\int_{\ov{\Omega}} m(x)\,  \dd \alpha(x) + \int_{\Omega}\beta(x)\cdot w(x)\, \dd x\right\}-\e $$
and so \eqref{mewnrenwrbw} follows by letting $\e \to 0$. 
Let us prove now that if $v \notin L^{  r}(E_1^m)^d$, then $\partial \cL_{r}(m, w)=\emptyset$.  We argue by contradiction supposing  that there exists $(\hat{\alpha},\hat{\beta}) \in \partial \cL_{r}(m, w)$. By \eqref{mewnrenwrbw} and  the assumption $w=0$ a.e. in $E_0^m$,  $\hat{\beta}$ must be a solution of the problem
\begin{equation}\label{problemaaux} \inf_{\beta \in   L^{r'}(\Omega)^d} J(\beta), \; \; \mbox{where } \;  J(\beta):= \int_{\Omega} \left[ \frac{1}{r'}|\beta|^{r'}- v\cdot \beta\right]m\,\dd x.\end{equation}
Since $vm=w \in L^{q}(\Omega)^d$ and $q\geq r$, we have that $J$ is Fr\'echet differentiable and 
\begin{equation}\label{qenqnbbdasaspqppqpw}0 = DJ(\hat{\beta})\beta= \int_{E_{1}^m} \left[ |\hat{\beta}|^{r'-2} \hat{\beta} -v\right]\cdot \beta m\, \dd x \hspace{0.3cm} \mbox{for all $\beta \in    L^{r'}(\Omega)^d$},\end{equation}
which implies that, since   $\beta$ is arbitrary  and $m>0$ on $E_1^m$,
\begin{equation}\label{qemmendbbdbas}v(x)=|\hat{\beta}(x)|^{r'-2} \hat{\beta}(x) \hspace{0.3cm} \mbox{for a.e. $x\in E_1^m$},\end{equation}
which  is a contradiction because  $|\hat{\beta}|^{r'-2} \hat{\beta} \in L^{r}(E_1^m)^{d}$. Now,   assume that $v \in L^{ r}(E_1^m)^d$ and let us prove that $\partial\cL_r(m,w)\neq \emptyset$. Define the functional $\hat{J}: L^{r'}_{m}(E_{1})^{d} \to \R $ (where $L^{r'}_{m}(E_{1}^m)$ denotes the space of measurable functions defined in $E_1^m$ which are integrable w.r.t. the measure $m$) as
$$\hat{J}(\beta) = \frac{1}{r'}  \int_{E_{1}^m} |\beta|^{r'} m\,\dd x    - \int_{E_{1}^m} v \cdot \beta  m\, \dd x.$$
Since $r'>1$, we have that $\hat{J}$ is coercive, continuous and strictly convex. Since $L^{r'}_{m}(E_{1}^m)^{d}$ is a reflexive Banach space, classical results in convex analysis imply the existence of a unique $\bar{\beta}\in L^{r'}_{m}(E_{1}^m)^{d}$ such that $\hat{J}(\bar{\beta})= \inf \{ \hat{J}(\beta) \; : \; \beta \in L^{r'}_{m}(E_{1}^m)^{d}\}$. The first order optimality condition implies that  $\bar{\beta}$ satisfies \eqref{qenqnbbdasaspqppqpw}-\eqref{qemmendbbdbas} and so
\begin{equation}\label{problemaaux2}\bar{\beta}(x)= |v(x)|^{r-2}v(x) \hspace{0.3cm} \mbox{for a.e. $x\in E_1^m$}.\end{equation}
Since $v\in L^r(E_1^m)^{d}$ we have that $\bar{\beta}  \in L^{r'}(E_1^m)^{d}$. Moreover, using that $L^{r'}(E_1^m)^{d} \subseteq L^{r'}_{m}(E_{1}^m)^{d}$, relation \eqref{mewnrenwrbw} implies that $ ( -|\bar{\beta}|^{r'}/r',\bar{\beta}) \in \partial\cL_r(m,w)$ and so $\partial\cL_r(m,w)\neq \emptyset $. Now, let $ (\alpha, \beta) \in \partial\cL_r(m,w)$.  The expression for $\ov{\cA_{r'}}$ in \eqref{expressionovA}  implies that $( -(1/r') |\beta|^{r'}, \beta)$ attains the supremum on the r.h.s. of \eqref{mewnrenwrbw}. Therefore, we must have
\begin{equation}\label{aasasasaadadas}\int_{\ov{\Omega}} m\, \dd \alpha + \frac{1}{r'} \int_{\Omega} |\beta|^{r'} m\, \dd x=\int_{E_1^m} m\, \dd \alpha + \frac{1}{r'} \int_{E_1^m} |\beta|^{r'} m\, \dd x=0.\end{equation}
Let us prove that $\alpha \mres E_1^m$ is absolutely continuous w.r.t. the Lebesgue measure restricted to $E_1^m$. Let  $B \in \mathcal{B}(E_1^m)$, 
 such that  $|B|=0$. Then, \eqref{aasasasaadadas} implies that
 $$\int_{E_1^m\setminus B} m\, \dd \alpha + \frac{1}{r'} \int_{E_1^m\setminus B} |\beta|^{r'} m\, \dd x+\int_{B} m\, \dd \alpha =0.$$
 By a standard argument using Lusin's theorem (to approximate the $\one_{E_1^m\setminus B}$ by continuous functions) and \eqref{expressionovA}   
we must have that $\int_{B} m \dd \alpha =0$ and since $m>0$ on $E_{1}^m$ we conclude that $\alpha(B)=0$. Thus $\alpha \mres E_1^m\ll |\cdot| \mres E_1^m$.  In particular, $\mbox{spt}(\a^{s}) \subseteq E_{0}^m$ and, denoting still by $\alpha$ the density of $\alpha$ restricted to  $E_1^m$,  $\alpha(x)+ (1/r')|\beta(x)|^{r'} \leq 0$ for a.e. $x\in E_{1}^m$. Therefore, by \eqref{aasasasaadadas} we have that 
 $$ \int_{E_{1}^m}m\left[ \alpha+ \frac{1}{r'}|\beta|^{r'}\right] \dd x =0, $$
 and since $m>0$ on $E_1^m$,  we conclude that $\alpha=-\frac{1}{r'}|\beta|^{r'}$ a.e. in $E_{1}^m$. Using \eqref{mewnrenwrbw} we get that $\beta$ solves problem \eqref{problemaaux} and so $\b= |v|^{r-2}v$ a.e. in  $E_1^m$ from which the result follows. 
\end{proof} 
 \begin{remark} Note that redefining the domain of $\cL_r$ as $C(\ov{\Om})\times L^{q}(\Om)^d$, the above proof   shows that the conclusions of the  Theorem \ref{subdifferentialcomputation} are still valid in this setting. 
 \end{remark}
%
%
\section{The optimization problem}\label{problemstatement}
In this entire section we suppose that  $q>d$. Let $f: \Omega \times\R \to \R$ be a  continuous function in both variables and  increasing  in the second variable.  Let us define the function 
$$\ds\Omega \times \R\ni (x,m) \mapsto F(x,m):=\int_0^m f(x,s)\, \dd s \in \R.$$
 Note that for every fixed $x\in \Omega$ the function $m\mapsto F(x, m)$ is convex.  Let us define
\begin{equation}\label{definitionF}\cF:W^{1,q}(\Omega)\to\overline{\R} \hspace{0.3cm} \mbox{as } \hspace{0.3cm}\ds\cF(m):=\int_\Omega F(x,m(x))\,\dd x.\end{equation}
 Given $w \in L^{q}(\Omega)^{d}$ we consider the following elliptic PDE 
\begin{equation}\label{principaleq} 
\left\{
\begin{array}{rcl}
-\Delta m + \diver(w)&=&0  \hspace{0.4cm} \iin\ \Omega,\\[4pt]
(\nabla m - w)\cdot n &=&0 \hspace{0.4cm}  \oon\ \partial\Omega.
\end{array}
\right.
\end{equation}
We say that $m\in W^{1,q}(\Omega)$ is a weak solution of  \eqref{principaleq} if
\begin{equation}\label{principalecweakform}
\int_\Omega \nabla m(x)\cdot \nabla\varphi(x)\, \dd x =\int_\Omega w(x) \cdot\nabla\varphi(x)\,\dd x \hspace{0.5cm} \forall \; \phi \in C^{1}(\overline{\Omega})
\end{equation}
 By Lemma \ref{div_surj} and Lemma \ref{regularity} in the Appendix for a given $w\in L^q(\Omega)^d$ equation \eqref{principaleq}   has a unique  solution $m\in W^{1,q}(\Omega)$ satisfying that $\ds\int_{\Omega} m(x)\,\dd x=1$. We consider the following optimization problem:
\begin{equation}
  \inf_{(m,w)\in\cK_P}\cJ_q(m,w):=\cL_q(m,w)+\cF(m), \tag{$P_q$}\label{Pq}
\end{equation}
where  the set of constraints $\cK_P$ is defined as
\small
$$\cK_P:=\left\{(m,w)\in W^{1,q}(\Omega)\times L^q(\Omega)^d \; ; \;  \mbox{such that $(m,w)$ satisfies \eqref{principaleq}}, \; \; \int_\Omega m(x)\,\dd x=1,\  m\le 1\right\}.$$\normalsize
\begin{remark} Since $\cL_q=\cL_q+\chi_{\{m\ge0\}}$,     the constraint $m\geq 0$ is implicitly imposed in  \eqref{Pq}.
\end{remark}
Given $s\in [1,\infty[$ we set $\ds W^{k,s}_\diamond(\Omega):=\left\{u\in W^{k,s}(\Omega):\int_\Omega u =0\right\}$.
Now, let us  define $A: \small W^{1,q}(\Omega)\to \left(W^{1,q'}_\diamond(\Omega)\right)^*$ and $B:L^q(\Omega)^d\to \left(W^{1,q'}_\diamond(\Omega)\right)^*$  \normalsize as 
$$ \llangle Am, \phi \rrangle  := \int_{\Omega} \nabla m(x)\cdot  \nabla \phi(x)\, \dd x, \hspace{0.3cm} \llangle Bw, \phi \rrangle := -\int_{\Omega} w(x)\cdot \nabla \phi(x)\, \dd x,$$
for all $m\in W^{1,q}(\Omega)$, $w\in L^q(\Omega)^d$ and $\phi \in  W^{1,q'}_\diamond(\Omega)$, 
where we used $\llangle \cdot, \cdot \rrangle$ to denote the duality product between $\left(W^{1,q'}_\diamond(\Omega)\right)^*$ and $W^{1,q'}_\diamond(\Omega)$.  Since $A$ and $B$ are linear  bounded operators,   the adjoint operators $A^*:\Wdqpom\to\Wqstarom$ and $B^*:\Wdqpom\to L^{q'}(\Omega)^d$ are well-defined and  given by
$$\langle A^*\phi,m\rangle=\int_\Omega\nabla\phi(x)\cdot \nabla m(x)\,\dd x, \hspace{0.6cm} \langle B^*\phi,w\rangle_{q',q} = -\int_\Omega \nabla\phi(x)\cdot  w(x)\,\dd x,$$
where we have used $\langle \cdot, \cdot \rangle$ to denote the duality product  between $ (W^{1,q}(\Omega))^{*}$ and $W^{1,q}(\Omega)$ and $\langle \cdot, \cdot \rangle_{q',q}$ to denote the duality product between $L^{q'}(\Om)^d$ and  $L^{q}(\Om)^d$.
  Now, let $I: W^{1,q}(\Omega) \to C(\overline{\Omega})$ be the Sobolev injection, which is well-defined since $q>d$ (see \cite{adams}),  and let $\cC:=\{ z \in C(\overline{\Omega}) \; ; \;  z  \leq 1\}$. Let us set  $X:=W^{1,q}(\Omega)\times L^q(\Omega)^d$, $Y:=\left(W^{1,q'}_\diamond(\Omega)\right)^*\times\R\times C(\overline{\Omega})$ and define the application $G:X\to Y$   as  
$$ G(m,w):=\left(Am+Bw, \int_\Omega m(x)\, \dd x-1, Im\right).$$
By setting $\cK:=\{0\}\times\{0\}\times \cC\subseteq Y$ we have that  $\cK_{P}$ can be rewritten as 
$$\cK_{P}= \{ (m,w) \in X \; : G(m,w) \in \cK  \}.$$
Since $A,$ $B$ and $I$ are linear bounded operators, we have that $ \cK_{P}$ is a closed and convex  subset of  $W^{1,q}(\Omega)\times L^q(\Omega)^d$.  
%
%
\begin{theorem}\label{existence} Problem \eqref{Pq} has (at least) one  solution $(m,w)$.
\end{theorem}

\begin{proof}
Since $|\Om|>1$ we have that $(m,w):=\ds\left(1/ |\Omega|,0\right)$ belongs to $\cK_{P}$     and the  cost function is finite. Now, let $(m_k,w_k)\in \cK_{P}$ be a   minimizing sequence.
 Since    $m\leq 1$ a.e. in $\Omega$ and $\cL_{q}(m_k, w_k)$ is bounded uniformly  in $k$, we get that $\|w_k\|_{L^q}$ is bounded.  Therefore,  there exists $w\in L^q(\Omega)^d$ such that, except for some subsequence,  $w_k \weakly w$ weakly in $L^q(\Omega)^d$. In addition, by Lemma \ref{regularity} and the boundedness of $w_k$ in $ L^q(\Omega)^d$ we have that  
$\nabla m_k$ is uniformly bounded in $L^q(\Omega)^d.$
Since $\ds\int_{\Omega} m_{k}(x)\, \dd x=1$,   Poincar\'e's inequality $\left(\left\|m_k-\frac{1}{|\Omega|}\right\|_{L^q}\le C\|\nabla m_k\|_{L^q}\right)$ implies that $m_k$ is bounded in  $W^{1,q}(\Omega)$. Thus,  there exists $m\in W^{1,q}(\Omega)$ such that, except for some subsequence, $m_k\weakly m$ weakly in $W^{1,q}(\Omega)$. Using these convergences, we get that $Am+ Bw=0$ and $\ds\int_{\Omega} m(x)\, \dd x=1$. 
%
The continuous embedding $I$ preserves the weak convergence and $\cC$ is weakly closed in $C(\ov{\Omega})$. Thus, $Im \in \cC$, which implies that $(m,w)\in \cK_P$. 
%
%
Since $\cJ_q$ is convex and l.s.c. w.r.t the weak topology in $W^{1,q}(\Omega) \times L^{q}(\Omega)^{d}$ (by Lemma \ref{subdiff_cost}) we get that $\cJ_q(m,w)= \inf \{\cJ_q(m_1,w_1) \; : \; (m_1,w_1)\in\cK_{P}\}$.  
\end{proof}

Now, we prove a {\it constraint qualification}  result for problem \eqref{Pq} (see e.g. \cite[Chapter 2]{BonSha}), which is crucial for deriving optimality conditions.  We set $ \mbox{{\rm dom}}(\cJ_q):=\{ (m,w) \in W^{1,q}(\Omega) \times L^{q}(\Omega)^{d} \; : \; \cJ_q(m,w) <\infty\}$. 
\begin{lemma}\label{qual_verified} We have that
\begin{equation}\label{qualificationproblemq}
0 \in \mbox{{\rm int}} \left\{   G(  \mbox{{\rm dom}}(\cJ_q)) - \cK \right\}.
\end{equation}

\end{lemma}
\begin{proof} We need to prove that for any given $(\d_1,\d_2,\d_3)\in Y$ small enough there exists $(m,w,c)\in \mbox{\rm{dom}}(\cJ_q)\times \cC$ such that 
\begin{equation}
Am+Bw =\d_1, \; \;  \; \; 
\int_\Omega m (x)\, \dd x=1+\d_2, \; \; \;  \; I(m)-c=\d_3.\tag{S}\label{S}
\end{equation}
We observe that $(m,0)\in \mbox{\rm dom}(\cJ_q)$, for all $m \in W^{1,q}(\Omega)\cap\mbox{\rm dom}(\cF)$ non-negative, which implies that we can search the solution of \eqref{S} in the form $(m,0,c)\in \mbox{\rm dom}(\cJ_q)\times \cC.$ First of all, note that for $\ds m_0:= 1/|\Omega|$ we have that 
$$ Am_0=0, \hspace{0.8cm} \int_\Omega m_0(x)\, \dd x=1,  \hspace{0.8cm} I m_0=m_0\in \mbox{\rm int}(\cC).
$$
By Lemma \ref{regularity} (see the Appendix),  there exists $m_1\in W^{1,q}(\Omega)$ such that 
\begin{equation}\label{equationforgamma} Am_1=\d_1 \hspace{0.8cm}  \; \; \mbox{and }  \hspace{0.8cm}  \; \int_\Omega m_1(x)\, \dd x=1+\d_2.\end{equation}
Setting  $\d m:=m_0-m_1,$ we obtain that $A\d m=-\d_1$ and $\ds \int_\Omega \d m(x)\, \dd x=-\d_2.$
By   Lemma \ref{regularity}   the linear bounded operator $\ds W^{1,q}(\Omega)\ni m \mapsto \left(Am, \int_{\Omega} m(x)\, \dd x\right) \in  \left(W^{1,q'}_\diamond(\Omega)\right)^*\times\R$ is surjective and so,  by the Open Mapping Theorem, there exists $h\in W^{1,q}(\Omega)$ such that 
$$ A h =0, \hspace{0.8cm}    \int_{\Omega} h(x)\, \dd x=0 \hspace{0.8cm}  \mbox{and } \hspace{0.8cm}  \|h-\d m\|_{W^{1,q}}=O\left(\|(\d_1,\d_2)\|_{\left(W_\diamond^{1,q'}(\Omega)\right)^*\times\R}\right).$$
 In particular, as $q>d$   the Sobolev inequality implies   that 
\begin{equation}\label{restestimate} \|I(h)-I(\d m)\|_{L^{\infty}}=O\left( \|(\d_1,\d_2)\|_{\left(W_\diamond^{1,q'}(\Omega)\right)^*\times\R}\right).\end{equation}
 Now, let $r:=h-\d m$ and for $\gamma >0$ let us define  $m_\g:=m_1+\g h$, which   by construction solves  \eqref{equationforgamma}. Since $m_0 \in \mbox{\rm int}(\cC)$,    if $\gamma$ is near to one (and  $\d_1,\d_2$ are small enough) then, by \eqref{restestimate}, $I(m_\g) =I(m_1)+\g I(\delta m) + \g I(r) \in \mbox{\rm int}(\cC)$. Thus, if $\| \delta_3\|_{L^\infty}$ is small enough we have that $c:= I(m_{\gamma})- \delta_{3}\in \mbox{\rm int}(\cC)$. Thus, $(m_\g, 0, c)$ solves \eqref{S} and $(m_\g, 0) \in  \mbox{{\rm dom}}(\cJ_q)$. The result follows.
\end{proof}
\section{Optimality conditions and characterization of the solutions}\label{optimalityconditions}
The purpose of this section is to derive optimality conditions for problem \eqref{Pq} and, as a consequence, to obtain the existence of solutions of $(MFG_{q})$.  As in the previous section we will assume in all the statements that $q>d$. Our strategy relies on a ``direct method'', which uses the constraint qualification condition established in Lemma \ref{qual_verified} and the characterization  of the subdifferential of $\cL_q$ (see Theorem \ref{subdifferentialcomputation}). 
The uniqueness and local regularity of the solutions are also discussed.  Moreover, in Subsection \ref{subsec:daul} we formulate the associated dual problem, and we provide an alternative (but related) argument to derive optimality conditions.

Let us define the \textit{Lagrangian} $\fL: \Wqom\times L^q(\Om)^d \times W^{1,q'}_{\diamond}(\Omega) \times\fM(\ov{\Omega})  \times   \R\to \overline{\R}$ as 
\begin{align*}\numberthis\label{lagrangian}
\fL(m,w,u,p,\l)&:=\cJ_q(m,w)-\left\llangle Am+Bw,u \right\rrangle+ \int_{\Omega} Im\, \dd p+\l\left( \int_\Omega m\,\dd x-1 \right)\\[2pt]
& = \cJ_q(m,w)-\left\langle A^*u-I^*p-\lambda,m \right\rangle- \int_{\Omega}B^*u \cdot w \dd x -\lambda.
\end{align*}

\begin{remark} 
Since the inclusion $W^{1,q}(\Omega)\hookrightarrow C(\overline{\Omega})$ is dense,    for every measure $p\in \fM(\ov{\Omega})$ the adjoint    of the injection operator $I^{*}p$ at  $p$ can be identified  uniquely with the restriction  of $p$ to $W^{1,q}(\Omega)$.  Thus, for notational convenience we will still write $p$ for $I^{*}p$. 
\end{remark}
Recall that for a Banach space $X$ and a convex closed subset $K\subseteq X$, the  {\it normal cone} to  $K$ at $x\in K$ is  defined as
\begin{equation}\label{cononormala}N_K(x):=\{x^*\in X^*:\langle x^*,z-x\rangle_{X^*,X}\le 0,\forall z\in K\},\end{equation}
where $\langle \cdot , \cdot \rangle_{X^*,X}$ denotes the duality pairing of $X^*$ and $X$. Using \cite[Example 2.63]{BonSha} we have
%
\begin{equation}\label{cononormalenG} N_{\cK}(G(m,w))= \left\{(u,\lambda, p) \in  W^{1,q'}_{\diamond}(\Omega)\times \R \times \fM_+(\ov{\Omega}) \; ; \; \mbox{{\rm spt}}(p)\subseteq \{m=1\} \right\}.
\end{equation}
Now, we provide the first order optimality conditions associated to a solution $(m,w)$ of \eqref{Pq}. 
\begin{theorem}\label{mfg1} Let $(m,w)\in \cK_P$ be a solution of problem \eqref{Pq}. Then,   $v:=(w/m)\one_{\{m>0\}}\in L^{q}(\Omega)^{d}$ and there exists $(u,p,\lambda)\in W^{1,s}(\Omega)\times \fM_+(\ov{\Omega}) \times \R$  for all $s\in ]1,d/(d-1)[$ and 
 $(\alpha,\beta)\in \partial\mathcal{L}_q(m,w)$ such that  $ A^{*}u\in \fM(\ov{\Omega})$ and the following optimality conditions hold true 
\begin{equation}\label{sis1}
\left\{\begin{array}{rcl} \alpha+ f(\cdot,m) - A^{*}u +  p +\lambda &=& 0, \\[4pt]
\beta&=& B^{*}u,\\[4pt]
Am+ Bw&=&0, \\[4pt]
\mbox{{\rm spt}}(p) \subseteq \{m=1\}, & 0\le m\le 1, & \ds\int_\Om m\,\dd x=1,\; \; \int_\Om u \dd x=0,
\end{array}\right.
\end{equation}
where   the first equality holds in $\fM(\ov{\Omega})$. Conversely, if there exists $(\alpha,\beta, u, p, \lambda)\in \partial\mathcal{L}_q(m,w)\times W^{1,q'}(\Omega)\times \fM_+(\ov{\Omega}) \times \R $ such that \eqref{sis1} holds true, then $(m,w)$ solves \eqref{Pq}.
\end{theorem}
\begin{proof} By Lemma \ref{qual_verified} the problem is qualified (see e.g. \cite[Chapter 2]{BonSha}). Thus, by classical results in convex analysis (see e.g.  \cite[Theorem 2.158 and Theorem 2.165]{BonSha}), we have the existence of $(u,p,\lambda)\in N_{\cK}(G(m,w))$  such that 
\begin{equation}\label{zerointhesubdi}\begin{array}{rcl} (0,0) \in \partial_{(m,w)} \fL(m,w,u,p,\lambda). \end{array}
\end{equation}
Since $ \cL_q$ is finite at $(1/|\Omega|,0)$   and the other terms   appearing in  $\fL$ are differentiable, by  \cite[Chapter 1, Proposition 5.6]{ekeland-temam}  and  \eqref{zerointhesubdi},  we must have that $\cL_{q}$ is subdifferentiable at $(m,w)$. Thus,    by Theorem \ref{subdifferentialcomputation} and \eqref{cononormala}   we get that $v\in L^{q}(\Omega)^{d}$ and there exists $(\alpha,\beta)\in \ov{\cA_{q'}}$ such that \eqref{sis1} holds true, with the first equation being an equality in $\Wqstarom$. Since, except by $A^{*}u$, all the other terms can be identified with  elements in $\fM(\ov{\Omega})$,  we have that  $A^{*}u$ can be identified to an element of  $\fM(\ov{\Omega})$. Using classical elliptic regularity theory (see \cite[Th\'eor\`eme 9.1]{stampacchia3}) we get that  $u\in W_\diamond^{1,s}(\Om)$ for any $s\in]1,d/(d-1)[$.  The fact that \eqref{sis1} is a sufficient condition follows also by the convexity of the problem (see \cite[Theorem 2.158]{BonSha}).
%
\end{proof}
As a corollary we  immediately obtain the following existence result for a  MFG type system with density constraints 
\begin{corollary}\label{mfgregularcase} There exists $(m, u,\mu, p,\lambda)\in W^{1,q}(\Om)\times W_{\diamond}^{1,s}(\Omega)\times \fM_+(\ov{\Omega})\times \fM_+(\ov{\Omega}) \times \R$  {\rm(}$s\in]1,d/(d-1)[${\rm)} such that
\begin{equation}
\left\{\begin{array}{rcl}
-\Delta u + \frac{1}{q'}|\nabla u|^{q'} +\mu - p - \l&=&f(x,m) \hspace{0.4cm} \iin\; \; \Om,\\[6pt]
\nabla u\cdot n&=&0 \hspace{0.4cm} \oon  \; \; \partial\Omega,\\[4pt]
-\Delta m - \diver\left(m |\nabla u|^\frac{2-q}{q-1}\nabla u\right) &=& 0 \hspace{0.4cm}\iin \; \; \Omega, \\[5pt]
\nabla m\cdot n&=&0  \hspace{0.4cm}  \oon  \; \; \partial\Omega,\\[4pt]
\ds\int_\Omega m\,\dd x=1, & \; &\ \ 0\le m\le 1 , \; \;  \iin\ \Omega,\\[6pt]
\mbox{{\rm spt}}(\mu) \subseteq \{m=0\},  & \; &   \mbox{{\rm spt}}(p) \subseteq \{m=1\},  
\end{array}\right.\tag{$MFG_q$}\label{MFGq}
\end{equation}
where the coupled system for $(u,m)$ is satisfied in the following weak sense: for all $\varphi \in C^{1}(\ov{\Om})$
\begin{equation}\label{interpretacionsistema}
\begin{array}{c}
\ds \int_\Om\nabla u\cdot\nabla\varphi\,\dd x +\int_{\Om}\frac{1}{r'}|\nabla u|^{r'}\varphi\,\dd x-\l\int_\Om\varphi\,\dd x +\int_{\ov{\Om}}\varphi\,\dd(\mu-p)=\int_\Om f(x,m(x))\varphi(x)\,\dd x,\\[6pt]
\ds \int_\Om\left(\nabla m+m |\nabla u|^\frac{2-q}{q-1}\nabla u\right)\cdot \nabla\varphi\,\dd x=0.
\end{array}
\end{equation}
\end{corollary} \smallskip
Let us define $E_2^m:=\left\{x\in\ov\Om: 0 < m(x) < 1\right\}$. Note that by the continuity of $m$, $E_2^m$ is an open set.
\begin{remark}[The uniqueness of the solutions] Assuming that the coupling $f$ is strictly increasing in its second variable, the objective functional in \eqref{Pq} becomes strictly convex in the $m$ variable (and the set $\cK_P$ is convex). Thus,  the function $m\in W^{1,q}(\Om)$ in \eqref{MFGq}  is unique, which implies also the uniqueness of  $w\in L^q(\Om)^d$. In particular,   $\nabla u\in L^{q'}(\Om)^d$ is also unique on $E_1^m$.   The first identity in \eqref{interpretacionsistema} with $\varphi\in C_c^1(E_2^m)$ implies   the uniqueness of  $\l\in\R$.  If   $\varphi\in C_c^1(E_1^m)$ we obtain 
$$\ds \int_\Om\nabla u\cdot\nabla\varphi\,\dd x +\int_{\Om}\frac{1}{r'}|\nabla u|^{r'}\varphi\,\dd x-\l\int_\Om\varphi\,\dd x - \int_\Om\varphi\,\dd p=\int_\Om f(x,m(x))\varphi(x)\,\dd x,$$
which together with the condition $\spt(p)\subseteq\{m=1\}$ yields the uniqueness of  $p$. Using   \cite[Theorem 3.4]{BarPor} we obtain that  on $E_1^m$ $u\in W^{1,q'}(\Om)$ is unique up to an additive constant which may differ on each connected component of the set $E_1^m$. In general, we cannot expect uniqueness for $\mu$.
%
%
%
\end{remark}

Now, let us discuss some interior regularity properties for the solutions on the open set $E_2^m$. Our approach is based on a bootstrapping argument. 

\begin{proposition}\label{regularity_boots} There exists  $\gamma_0 \in ]0,1[$ such that 
\begin{equation}\label{reg1} u\in C^{1,\g_0}_{\rm{loc}}(E^m_2)\;\;\; {\rm{and}}\;\;\; m\in C^{1,\g_0}_{\rm{loc}}(E^m_2).\end{equation}
If in addition, $f\in C_{loc}^{j,\g_1}({\ov\Om}\times\R)$ for $j\in\{0,1\}$ and some $\gamma_1 \in ]0,1[$, we have that for some $\gamma_2 \in ]0,1[$
\begin{equation}\label{reg2}u\in C^{2+j,\g_2}_{\rm{loc}}(E_2^m).\end{equation}
\end{proposition} 

\begin{proof}
{\it Step 1.} We show that there exists $k>d$ such that $u\in W_{\rm{loc}}^{2,k}(E_2^m)$.  By the classical Sobolev embeddings, this implies  that $u\in C_{\rm{loc}}^{1,\g}(E_2^m)$ (for some $\gamma \in ]0,1[$).   Let  $r_1\in ]q',d/(d-1)[$. Since $u\in W_\diamond^{1,r_1}(\Om)$   we have  that $|\nabla u|^{q'}\in L^{r_1/q'}(\Om)$. The continuity of $f$ and the density constraint on $m$ imply that $f(\cdot, m(\cdot))\in L^{\infty}(\Omega)$. Thus, denoting by $\d_1:=r_1/q'$, classical regularity theory (see \cite{gilbarg}) yields $u\in W_{\rm{loc}}^{2,\d_1}(E_2^m)$. In particular,    the Sobolev inequality (see e.g. \cite{adams}) yields $u\in W_{\rm{loc}}^{1,\frac{d\d_1}{d-\d_1}}(E_2^m)$ and so $|\nabla u|^{q'}\in L_{\rm{loc}}^{\frac{d\d_1}{q'(d-\d_1)}}(E_2^m)$. We easily check that $\delta_2:= d\delta_1/q'(d -\delta_1)>\delta_1$ 
%
%
%
 and so $u\in W^{2,\d_2}_{\rm{loc}}(E_2^m)$.
 Let us define  the sequence $\d_{i+1}:=  \frac{d\d_i}{(d-\d_i)q'}$. Since $\delta_{i+1}-\delta_{i} \geq (q'+d-dq')/(d-\delta_i)q'$ and $q'+d-dq'>0$ , after a finite number of steps we get the existence of $  i^*\ge 2$ such that $k:=\d_{i^*}>d$ and $u\in W_{\rm{loc}}^{2,k}(E_2^m)$.   \smallskip\\
{\it Step 2.} Let us prove that $m\in C_{\rm{loc}}^{1,\g_0}(E_2^m)$ for some $\gamma_0 \in ]0,1[$. Since $m\in W^{1,q}(\Om)$ and $q>d$, we already have that $m$ is  H\"older continuous.   Having $u\in C_{\rm{loc}}^{1,\g}(E_2^m)$, this implies that  $\nabla u\in C_{\rm{loc}}^{0,\g}(E_2^m)^d,$ hence $m|\nabla u|^{\frac{2-q}{q-1}}\nabla u\in C_{\rm{loc}}^{0,\hat{\gamma}}(E_2^m)^d$, for some $\hat{\gamma}\in ]0,1[$. Using a Schauder-type estimate (see  \cite[Theorem 5.19]{giaquinta-martinazzi}) we get that $m\in C_{\rm{loc}}^{1,\gamma'}(E_2^m)$ for some $\gamma' \in ]0,1[$. 

{\it Step 3.}  Using the above regularity for $m$, if $f\in C_{\rm{loc}}^{0,\gamma_1}(\ov{\Omega}\times \R)$, the   local H\"older regularity for  $\frac{1}{q'}|\nabla u|^{q'}$ and \cite[Corollary 6.9]{gilbarg} imply that $u \in  C_{\rm{loc}}^{2,\g''}(E_2^m)$ for some $\gamma''\in ]0,1[$. Finally,  if $f\in C_{\rm{loc}}^{1,\gamma_1}(\ov{\Omega}\times \R)$, the local H\"older regularity of $\nabla m$ and of $D^{2}u$ imply that  $u \in  C_{\rm{loc}}^{3,\g'''}(E_2^m)$ for some $\gamma'''\in ]0,1[$.

\end{proof}
%
%

\subsection{The dual problem}\label{subsec:daul}
In order to write explicitly the dual problem we will need the following Lemma concerning the Legendre-Fenchel transform of $\cF.$ 
\begin{lemma}\label{cF_transform}
Let $\cF$ be defined by \eqref{definitionF}. Then its Legendre-Fenchel transform $\cF^*:\Wqstarom\to\overline{\R}$ is given by 
\begin{equation}\label{definiciondeFmay}
\cF^*(m^*)=\left\{
\begin{array}{ll}
\ds\int_\Omega F^*(x,m^*(x))\,\dd x, & \iif\ m^*\in\fM_{ac}(\ov{\Omega}),\\
+\infty, & \rm{otherwise},
\end{array}
\right.
\end{equation}
where $F^*$ denotes the Legendre-Fenchel transform of $F$ w.r.t. the second variable.
\end{lemma}
\begin{proof}
The result is a consequence of  \cite[Section 2]{brezis}. 
\end{proof}
We recall that given a Banach space $X$ and a convex closed set $K\subseteq X$,  the  {\it support function}  $\s_K:X^*\to\ov\R$ is defined as
$$\s_K(x^*):=\sup_{x\in K}\langle x^*,x\rangle_{X^*,X} \; \hspace{0.5cm} \forall x^*\in X^*.$$
\begin{proposition}\label{dual} The dual problem of \eqref{Pq} (in the sense of convex analysis) has at least one solution and can be written as 
\begin{equation}
-\min_{(u,p,\l,a)\in\cK_D}\left\{\int_\Omega F^*(x,a)\,\dd x+\l+p(\ov{\Omega})\right\} \tag{$PD_q$}\label{PDq}
\end{equation}
where
$$\cK_D:=\left\{(u,p,\l,a)\in\Wdqpom\times  \fM_+(\ov{\Omega}) \times\R\times \fM_{ac}(\ov{\Omega}):\ A^* u +\frac{1}{q'}|B^* u|^{q'} - p - \l \le a\right\}$$ where the inequality has to be understood in the sense of  measures. 
\end{proposition}


\begin{proof} The dual problem of \eqref{Pq} can be written as
\begin{equation}\label{primerdual}\max_{\begin{subarray}{c}
(u,p,\l)\in\\
\Wdqp(\Omega)\times\fM(\ov{\Omega})\times\R
\end{subarray}
}\left\{\inf_{\begin{subarray}{c}
(m,w)\in\\
W^{1,q}(\Omega)\times L^q(\Omega)^{d}
\end{subarray}
}\fL(m,w,u,p,\l)-\s_{\cK}(u,\lambda,p)\right\} ,
\end{equation}
where $\fL$ is defined in \eqref{lagrangian} and we recall that $\cK:=\{0\}\times\{0\}\times \cC$. The fact that we have a $\max$ instead of a $\sup$ in \eqref{primerdual} is justified by Lemma  \ref{qual_verified} and  \cite[Theorem 2.165]{BonSha}.  Now, note that 
\begin{equation}\label{expressionfuncionsop}\s_{\cK}(u,\lambda,p)= \s_{\cC}(p)=\left\{\begin{array}{ll}  p(\ov{\Omega}) & \mbox{if } p  \in \fM_{+}(\ov{\Omega}), \\[4pt]
+\infty  & \mbox{otherwise.}\end{array}\right.\end{equation}
%
On the other hand,  we have that 
\begin{align*}
\inf_{(m,w)}\fL(m,w,u,p,\l)&=-\sup_{(m,w)}-\fL(m,w,u,p,\l),\\
& = -\sup_{(m,w)}\left\{\left\langle A^*u - p -\l, m \right\rangle + \int_{\Omega}B^*u \cdot w \dd x -\cJ_q(m,w)\right\}-\l,\\
& = - \cJ_q^*(A^*u-p-\l,B^*u)-\l.
\end{align*}
Since there exists $(m,w)\in \dom(\cL_q)$ at which $\ds\cF$ is continuous (take for example $(m,w)=(1/|\Om|,0)$),   for any $(\a,\b)\in\Wqstarom\times L^{q'}(\Omega)^d$   we have that (see e.g. \cite[Theorem 9.4.1]{attouch})
\begin{equation}\label{expressionconjugate}
\begin{array}{ll}
\cJ_q^*(\a,\b)&=\left(\cL_q+\cF\right)^*(\a,\b),\\[4pt]
&= \inf_{a\in\Wqstarom}\left\{\cL_q^*(\a-a,\b)+\cF^*(a)\right\},\\[4pt]
&= \inf_{a\in\Wqstarom}\left\{\chi_{\overline{\cA_{q'}}}(\a-a,\b)+\cF^*(a)\right\},\\
&= \inf_{a\in\fM_{ac}(\ov{\Omega})}\left\{\ds \int_\Omega F^*(x,a(x))\,\dd x:\ \a+\frac{1}{q'}|\b|^{q'}\le a\right\},
\end{array}
\end{equation}
where   we have used Lemma \ref{subdiff_cost} and Lemma \ref{cF_transform}. 
Let us prove that the above minimization problem has a solution. First,  by \eqref{definiciondeFmay} the integral functional is l.s.c. with respect to the weak$-\star$ topology of measures. Let us take a minimizing sequence $a_n\in L^1(\Omega).$ There exists a    constant $C>0$ such that 
$$C\ge \int_\Omega F^*(x,a_n(x))\,\dd x\ge \int_\Omega \left[a_n(x) y(x)-F(x,y(x))\right]\dd x,\;\;\forall y\in L^{\infty}(\Omega).$$
By choosing $y(x)= {\rm sgn}(a_n(x))$ (which is equal to $1$ if $a_n(x)\geq 0$ and $-1$ if not), we obtain that $a_n$ is   bounded in $L^1(\Omega)$. 

Therefore, when the sequence $a_n$ is identified to a sequence of measures, we get a  weakly$-*$ convergent subsequence to   some $a\in\fM(\ov{\Omega}).$ The constraint is convex and closed with respect to this convergence, so by the lower semicontinuity  of the objective functional we have that $a$ is a solution and,  by Lemma \ref{cF_transform},  $a\in \fM_{ac}(\ov{\Omega})$ as well.  Using this result and \eqref{primerdual}, \eqref{expressionfuncionsop} and \eqref{expressionconjugate}, the conclusion follows.  
\end{proof}

Using the dual problem, let us provide an alternative, but related  way, to obtain first order necessary and sufficient optimality conditions. By  Theorem \ref{existence} and Proposition \ref{dual} we know that that there exist $(m,w)\in \cK_P$ and $(u,p,\l,a)\in\cK_D$ optimizers for \eqref{Pq} and \eqref{PDq} respectively. Moreover, since Lemma \ref{qual_verified} implies that  problem \eqref{Pq} is qualified, by  \cite[Theorem 2.165]{BonSha} problem \eqref{PDq} has the same value as problem \eqref{Pq}. Therefore,
\begin{equation}\label{condicionprimaligualdual}\begin{array}{c}\cL_q(m,w)+\cF(m)=- \ds \int_\Omega F^*(\cdot,a)\,\dd x - \l - \s_\cC(p), \\[4pt]
\ds Am + Bw=0, \; \; \;  m\le 1,\;\;\; \int_\Om m\,\dd x=1,\;\;\; A^*u-p-\l-a+\frac{1}{q'}|B^*u|^{q'}\le 0.\end{array}\end{equation}
Using the above relations, we obtain
\begin{align*}
0&=\cF(m)+\int_\Omega F^*(\cdot,a)\,\dd x+\cL_q(m,w)+\chi_{\overline{\cA_{q'}}}(A^*u-p-\l-a,B^*u)+\l+\s_\cC(p),\\[4pt]
&=\cF(m)+\int_\Omega F^*(\cdot,a)\,\dd x+\cL_q(m,w)+\cL_q^*(A^*u-p-\l-a,B^*u)+\l+\s_\cC(p),\\[4pt]
&\ge \langle a,m\rangle_{\fM(\ov{\Omega}),C(\overline{\Omega})}+\cL_q(m,w)+\cL_q^*(A^*u-p-\l-a,B^*u)+\l\int_\Omega m\,\dd x+\s_\cC(p),\\[4pt]
&\ge \langle a,m\rangle_{\fM(\ov{\Omega}),C(\overline{\Omega})} + \langle A^*u-p-\l-a, m\rangle_{\fM(\ov{\Omega}),C(\overline{\Omega})} + \int_{\Omega} B^*u \cdot w\, \dd x+\l\int_\Omega m\,\dd x+\s_\cC(p),\\[4pt]
&\ge \langle A^*u-p, m\rangle_{\fM(\ov{\Omega}),C(\overline{\Omega})} + \int_{\Omega} B^*u \cdot w\, \dd x+\langle p,m\rangle_{\fM(\ov{\Omega}),C(\overline{\Omega})},\\[4pt]
& = \langle A^*u,m \rangle + \langle B^*u,w \rangle_{q',q},\\[4pt]
& =  \left\llangle Am,u \right \rrangle +  \left\llangle Bw,u \right \rrangle=0.
\end{align*}

This means that all the inequalities in the previous list are actually equalities. Thus, 
\begin{itemize}
\item[(i)] $\cF(m)+\cF^*(a)=\langle a, m\rangle$ and so,  using the fact that $\cF$ is differentiable on $W^{1,q}(\Omega)$, we have $a=f(\cdot,m)$.  \vspace{0.2cm}
\item[(ii)] $\cL_q(m,w)+\cL_q^*(A^*u-p-\l-a,B^*u)=\langle A^*u-p-\l-a, m\rangle  + \langle B^*u,w \rangle_{q',q},$
namely 
$$\left(A^*u-p-\l-a,B^*u \right) \in\partial \cL_q(m,w)$$
\item[(iii)] $\ds \s_\cC(p) =\langle p,m\rangle_{\fM(\ov\Omega),C(\overline{\Omega})},$
which implies that  $p\in N_\cC(m)$. 
\end{itemize} \vspace{0.3cm}

Using \eqref{cononormalenG}, \eqref{condicionprimaligualdual} and {\rm(i)}-{\rm(iii)} we recover system  \eqref{sis1}.

%
%
%
%
%

\section{Treating less regular cases via an approximation argument}\label{aproximationargument}
In this section we  provide   the proof of  the existence of a solution of a suitable form of \eqref{MFGq} when $1<q\le d$.  Note that given $w\in L^{q}(\Omega)^d$ the solution $m$ of \eqref{principaleq} is in general discontinuous. Because of the constraint $0\leq m \leq 1$, this implies that  problem \eqref{Pq} is in general not qualified (see \cite[Chapter 2]{BonSha}) and thus the arguments in the previous section are no longer valid. 
%
%
%
To handle this issue, we propose an approach which is based on a regularization argument.

Let us fix  $1<q\le d$ and $r>d$. For $\e>0$   define $\cJ_{q,\e}: W^{1,r}(\Om)\times L^r(\Omega)^d\to\overline{\R}$ as
$$\cJ_{q,\e}(m,w):=\cJ_q(m,w)+\e\cL_r(m,w).$$
Following the arguments in the proof of Theorem \ref{existence},   problem 
\begin{equation}
\inf_{(m,w)\in\cK_P}\cJ_{q,\e}(m,w)\tag{$P_{q,\e}$}\label{Pqe}
\end{equation}
admits at least one solution $(m_\e,w_\e)$. Since $m_\e\in C(\ov{\Omega})$, problem  \eqref{Pqe} is qualified. Moreover, since  both $\cL_q$ and $\cL_r$ are continuous at $(\hat{m}, \hat{w}):=(1/|\Om|,0)$, by \cite[Chapter 1, Proposition 5.6]{ekeland-temam} we have that 
$$\partial(\cL_q(m,w)+\e\cL_r(m,w))=\partial\cL_q(m,w)+\e\partial\cL_r(m,w) \hspace{0.2cm} \mbox{for all $(m,w)\in W^{1,r}(\Om)\times L^r(\Om)^d.$}$$ 
Therefore, exactly as in the proof of Theorem \ref{mfg1}, if we define  $v_\e:=(w_\e/m_\e)\one_{E_1^{m_\e}}$ we have that  $v_\e \in L^r(\Om)^d$ and there exist $(u_\e,p_\e,\l_\e)\in W^{1,s}_{\diamond}(\Omega)\times \fM_+(\ov{\Omega}) \times \R$ ($s\in]1,d/(d-1)[$),  $(\a_{\e,q},\b_{\e,q})\in\partial\cL_q(m_\e,w_\e)$ and $(\a_{\e,r},\b_{\e,r})\in\partial\cL_r(m_\e,w_\e)$ such that 
\begin{equation}\label{opt_eps1}
\left\{
\begin{array}{rcl}
\a_{\e,q} + \e\a_{\e,r} - A^* u_\e + f(x,m_\e) + p_\e + \l_\e & =  & 0,\\[4pt]
\b_{\e,q}+\e\b_{\e,r} & = & B^* u_\e,\\[4pt]
A m_\e + B w_\e & = & 0,\\[4pt]
\ds\int_\Om m_\e\,\dd x = 1,& 0\le m_\e\le 1, & {\rm{spt}}(p_\e)\subseteq \{m_\e=1\}.
\end{array}
\right.
\end{equation}
%
%
%
Now, for $\e \geq 0$, let us define
$F_{q,\e},G_{q,\e},H_{q,\e}:\R^d\to\R$ as
$$F_{q,\e}(z):=\frac{1}{q}|z|^q+\frac{\e}{r}|z|^r,\;\;\; G_{q,\e}(z):=\frac{1}{q'}|z|^q+\frac{\e}{r'}|z|^r,\;\;\;{\rm{and}}\;\;\; H_{q,\e}(z):=G_{q,\e}(\nabla F_{q,\e}^*(z)).$$
For notational convenience, we set   $H_{q}:= H_{q,0}$. Elementary arguments in convex analysis show that $H_{q,\e} \to H_{q}$ uniformly over compact sets. System \eqref{opt_eps1} can be written in the following alternative form:
%

\begin{proposition}
There exists $\tilde\a_\e\in\fM_-(\ov\Om)$ such that 
\begin{equation}
\left\{
\begin{array}{rcl}
\vspace{5pt}
-\Delta u_\e + H_{q,\e}(-\nabla u_\e) - p_\e - \tilde\a_\e - \l_\e&=&f(x,m_\e), \;\; \iin\ \Om,\\
\vspace{5pt}
-\Delta m_\e + \diver\left(m_\e\nabla F_{q,\e}^*(-\nabla u_\e)\right) &=& 0, \;\; \iin\ \Omega, \\[4pt]
\nabla m_\e\cdot n = 0 &\;& \nabla u_\e\cdot n =0, \;\; \oon\ \partial\Omega.\\[4pt]
 0\le m_\e\le 1, & \; &\ \ds\int_\Om m_\e\,\dd x=1,\\[4pt]
\mbox{{\rm spt}}(p_\e) \subseteq \{m_\e=1\},   &  \; &{\rm{spt}}(\tilde\a_\e)\subseteq\{m_\e=0\}.
\end{array}
\right.\tag{$MFG_{q,\e}$}\label{MFGqe}
\end{equation}
\end{proposition}

\begin{proof}
By  Theorem \ref{subdifferentialcomputation} we have that 
$$\a_{\e,q}\mres E_1^{m_\e}=-\frac{1}{q'}|v_\e|^q,\;\;\; \a_{\e,r}\mres E_1^{m_\e}=-\frac{1}{r'}|v_\e|^r, \; \; \b_{\e,q}\mres E_1^{m_\e}=|v_\e|^{q-2}v_\e,\;\;\; \b_{\e,r}\mres E_1^{m_\e}=|v_\e|^{r-2}v_\e.$$
On the other hand, since $\nabla u_\e \in L^{r'}(\Omega)^{d}$ we have that $ \ov v_\e:=\nabla F_{q,\e}^*(-\nabla u_\e)\in L^r(\Om)^d$.  Using that $\nabla F_{q,\e}(v_\e)= \b_{\e,q}+\e\b_{\e,r}=-\nabla u_\e$ in $ E_1^{m_\e}$ and that $\nabla F_{q,\e}^{-1}= \nabla F_{q,\e}^{*}$, we get that $\ov v_\e=  v_\e$ in $ E_1^{m_\e}$. Therefore,  
%
there exists $\xi_\e\in L^{q'}(\Om)^d$ such that $\spt(\xi_\e)\subseteq E_0^{m_\e}$ and a.e. in $\Om$
$$\b_{\e,q}=|\ov v_\e|^{q-2}\ov v_\e+\xi_\e\;\;\;{\rm{and}}\;\;\; \b_{\e,r}=\frac{1}{\e}( \nabla F_{q,\e}(\ov v_\e)-\b_{\e,q})= |\ov v_\e|^{r-2}\ov v_\e-(1/\e)\xi_\e.$$
Using the convexity of $\frac{1}{q'}|\cdot|^{q'}$ and  $\frac{1}{r'}|\cdot|^{r'}$,  we easily check  that 
$$\frac{1}{q'}|\b_{\e,q}|^{q'}\ge \frac{1}{q'}|\ov v_\e|^q + \ov v_\e\cdot\xi_\e \hspace{0.2cm} \mbox{and }  \hspace{0.2cm} \frac{\e}{r'}|\b_{\e,r}|^{r'}\ge \frac{\e}{r'}|\ov v_\e|^r - \ov v_\e\cdot\xi_\e.$$
Hence 
$$-\frac{1}{q'}|\b_{\e,q}|^{q'}-\frac{\e}{r'}|\b_{\e,r}|^{r'}\le - \frac{1}{q'}|\ov v_\e|^q- \frac{\e}{r'}|\ov v_\e|^r=-H_{q,\e}(-\nabla u_\e),$$
with an equality a.e. in $E_1^{m_\e}.$ In particular, we have the existence of a positive measure $\g_\e$ such that $\spt(\g_\e)\subseteq \spt(\xi_\e)\subseteq E_0^{m_\e}$ and
$$-\frac{1}{q'}|\b_{\e,q}|^{q'}-\frac{\e}{r'}|\b_{\e,r}|^{r'}= -H_{q,\e}(-\nabla u_\e)-\g_\e.$$
Since the definition of $(\a_{\e,q},\b_{\e,q})$ and $(\a_{\e,r},\b_{\e,r})$ implies the existence of  two  positive measures $\tilde\a_{\e,q}$ and $\tilde\a_{\e,r}$ such that $\spt(\tilde\a_{\e,q})\subseteq E_0^{m_\e},$ $\spt(\tilde\a_{\e,r})\subseteq E_0^{m_\e}$ and 
$$\a_{\e,q}=-\frac{1}{q'}|\b_{\e,q}|^{q'}-\tilde\a_{\e,q}\;\;\; {\rm{and}}\;\;\; \a_{\e,r}=-\frac{1}{r'}|\b_{\e,r}|^{r'}-\tilde\a_{\e,r},$$
the result follows by setting   $\tilde\a_\e:=-\tilde\a_{\e,q}-\e\tilde\a_{\e,r}-\g_\e$.
\end{proof}

Now we present the main theorem of this section.

\begin{theorem}\label{main_appr}
There exists $(m,u,p,\mu,\l)\in  W^{1,q}(\Om)\times W^{1,q'}_\diamond(\Om)\times\fM_+(\ov\Om)\times\fM_+(\ov\Om)\times\R$  such that 
\begin{equation}
\left\{
\begin{array}{rcl}
-\Delta u + \frac{1}{q'}|\nabla u|^{q'} +\mu - p - \l&=&f(x,m) \hspace{0.4cm} \iin\; \; \Om,\\[6pt]
\nabla u\cdot n&=&0 \hspace{0.4cm} \oon  \; \; \partial\Omega,\\[4pt]
-\Delta m - \diver\left(m |\nabla u|^\frac{2-q}{q-1}\nabla u\right) &=& 0 \hspace{0.4cm}\iin \; \; \Omega, \\[5pt]
\nabla m\cdot n&=&0  \hspace{0.4cm}  \oon  \; \; \partial\Omega,\\[4pt]
\ds \int_\Omega m\,\dd x=1, & \; &\ \ 0\le m\le 1 , \; \;  \iin\ \Omega,\\[6pt]
\end{array}\right.\tag{$MFG_q$}\label{MFGq5}
\end{equation}
where the coupled system for $(u,m)$ is satisfied in the  weak sense {\rm(}see \eqref{interpretacionsistema}{\rm)}. Moreover, defining 
\begin{equation}\label{definiciondualidad}\langle \mu-p,m\rangle:=\lambda+ \int_{\Om} \left[f(x,m)-\frac{1}{q'}|\nabla u|^{q'}\right] m\, \dd x-\int_{\Om} \nabla m \cdot \nabla u \, \dd x\end{equation}
we have the inequality 
\begin{equation}\label{weakconcentrationproperty}
 \int_\Om\dd p+\langle \mu-p,m\rangle\leq 0.
\end{equation}
%
\end{theorem}
\begin{proof} 
{\it Step 1:  Bounds for   $\l_\e,$ $p_\e$ and $\tilde\a_\e$.} Note first that the second equation in \eqref{opt_eps1} and Theorem \ref{subdifferentialcomputation} imply that $w_\e= m_\e v_\e$ a.e. in $\Omega$. Also, in the set $E_{1}^{m_\e}$ we have that $\nabla F_{q,\e}(v_\e)=-\nabla u_\e$ and so in $E_{1}^{m_\e}$ the identities $v_\e= \nabla F_{q,\e}^{*}(-\nabla u_\e)$ and $ H_{q,\e}(-\nabla u_\e)=G_{q,\e}(v_\e)$ hold true. Now, by the second and third equations in \eqref{opt_eps1} we get that 
\begin{equation}\label{nice_iden}
\ds\int_\Om\nabla u_\e\cdot\nabla m_\e\,\dd x=\int_\Om\nabla u_\e\cdot w_\e\,\dd x=-\int_\Om m_\e\nabla F_{q,\e}(v_\e)\cdot v_\e\,\dd x=-\int_\Om m_\e(|v_\e|^q+\e|v_\e|^r)\,\dd x
\end{equation}
and so taking $m_\e$ as test function in the first equation of \eqref{opt_eps1}, we obtain  
\begin{align*}
\l_\e+\int_\Omega m_\e\,\dd p_\e & =\int_\Omega\left( G_{q,\e}(v_\e)m_\e+\nabla u_\e\cdot\nabla m_\e-f(x,m_\e)m_\e\right)\,\dd x\\
& =\int_{\Omega}\left(-\frac{1}{q}|v_\e|^qm_\e-\e\frac{1}{r}|v_\e|^rm_\e-f(x,m_\e)m_\e\right)\,\dd x,
\end{align*}
which implies that 
\begin{equation}\label{identidad1}
\l_\e+\int_\Omega m_\e\,\dd p_\e=-\cL_q(m_\e,w_\e) - \e\cL_r(m_\e,w_\e) -\int_{\Omega}f(x,m_\e)m_\e\,\dd x.
\end{equation}
The optimality of $(m_\e,w_\e)$ yields
\begin{equation}\label{uniformboundsdifferencelq}0\leq \cL_q(m_\e,w_\e) + \e\cL_r(m_\e,w_\e) \leq \cJ_{q,\e}(1/|\Om|,0)- \cF(m_\e).\end{equation}
Thus, since $f$ is continuous, $0\leq m_\e \leq 1$,   \eqref{identidad1} and the fact that $\mbox{{\rm spt}}(p_\e) \subseteq \{m_\e=1\}$ yield the existence of a constant $c_1>0$ (independent of $\e$) such that 
\begin{equation}\label{bound1}-c_1 \leq \l_\e+\int_\Omega m_\e\,\dd p_\e = \l_\e+ |p_\e|_{TV}\leq c_1.\end{equation}
On the other hand, by taking $1-m_\e$ as test function in the first equation of \eqref{opt_eps1}, a similar computation using  \eqref{nice_iden} yields 
\begin{equation}\label{alltheterms}(|\Omega|-1)\l_\e -  |\tilde\a_\e|_{TV}=\int_\Omega H_{q,\e}(-\nabla u_\e)\,\dd x+\cL_q(m_\e,w_\e) + \e\cL_r(m_\e,w_\e)-\int_\Omega f(x,m_\e)(1-m_\e)\,\dd x,\end{equation}
from which 
\begin{equation}\label{bound2}
 (|\Omega|-1)\l_\e - |\tilde\a_\e|_{TV} \geq -\int_\Omega f(x,m_\e)(1-m_\e)\,\dd x \geq c_2,\end{equation}
where $c_2>0$ is independent of $\e$. Since $|\Omega|>1$, inequalities \eqref{bound1}-\eqref{bound2} imply that $\lambda_\e$ is uniformly bounded w.r.t. $\e$ and so $p_\e$ and  $\tilde\a_\e$ are uniformly bounded w.r.t. $\e$ in $\fM(\ov\Om)$.\smallskip\\
{\it Step 2: Convergence of  $\nabla u_\e$ and  $m_\e$.} By \eqref{alltheterms}, as a function of $\e$ we have  that $H_{q,\e}(-\nabla u_\e)$ is uniformly bounded  in $L^1(\Omega)$ which implies that $u_\e$ is bounded in $W^{1,q'}(\Omega)$ and that  $-\Delta u_\e$
 is  bounded  in $\fM(\overline{\Omega})$. On the one hand, the boundedness of $u_\e$  in $W^{1,q'}(\Omega)$ implies the existence of $u\in W^{1,q'}(\Omega)$ such that up to some subsequence $u_\e$ converges weakly to $u$ in $W^{1,q'}(\Omega)$. In particular, $\ds\int_\Om u\, \dd x=0$. On the other hand, the boundedness of $-\Delta u_\e$ in  $\fM(\overline{\Omega})$ and \cite[Theorem  1.3 with $p=2$]{mingione2} 
 imply  the existence of $ s\in]0,1[$ and $\d_0>0$ such that $\nabla u_\e$ is uniformly bounded in $W^{s,1+\d_0}_{\rm{loc}}(\Om)^d$. By  \cite[Corollary 7.2]{frac_sob}  we can extract a subsequence such that $\nabla u_\e\to\nabla u$ a.e. in $\Om$   and so $H_{q,\e}(-\nabla u_\e)\to \frac{1}{q'}|\nabla u|^{q'}$ a.e. in $\Omega$.  

Now, in order to establish the convergence for $m_\e$, note that inequality \eqref{uniformboundsdifferencelq} and the fact that $0\leq m_\e \leq 1$ imply that 
$w_\e$ is uniformly bounded in $L^q(\Omega)^d$ for all $\e>0.$ This means that, up to some subsequence, $w_\e$ is converging weakly in $L^q(\Om)^d.$  Since Lemma \ref{regularity} implies that  $\|\nabla m_\e\|_{L^q}\le C\|w_\e\|_{L^q}$  (for a constant $C>0$ independent of $\e$), by Poincar\'e's inequality we get that $m_\e$ is uniformly bounded in $W^{1,q}(\Om)$. Extracting a subsequence again, there exists $m$ such that $m_\e$ converges weakly to $m$ in $W^{1,q}(\Om)$. By the compact Sobolev embedding, we get strong convergence in $L^q(\Omega)$, which implies that
 $0\le m \le 1$ a.e. in $\Om$ and $\ds\int_\Omega m \; \dd x=1$.\smallskip\\
{\it Step 3: The limit equations.} 
The weak formulation of the second equation in \eqref{MFGqe}  yields
$$\int_\Om\nabla m_\e\cdot\nabla\varphi\,\dd x=-\int_\Om m_\e\nabla F^*_{q,\e}(-\nabla u_\e)\cdot\nabla\varphi\,\dd x, \hspace{0.3cm} \mbox{for all $\varphi\in C^\infty(\ov\Om).$}$$
Since, extracting a subsequence,  $w_\e=m_\e\nabla F^*_\e(-\nabla u_\e)$  converges weakly in $L^q(\Om)^d$ to some $w$,  the weak convergence of $m_\e$ in $W^{1,q}(\Om)$ implies that 
$$\int_\Om\nabla m \cdot\nabla\varphi\,\dd x=\int_\Om w\cdot\nabla\varphi\,\dd x, \hspace{0.3cm} \mbox{for all $\varphi\in C^\infty(\ov\Om).$}$$
Moreover, extracting a subsequence again,  we get  that 
 $$m_\e(x)\nabla F^*_{q,\e}(-\nabla u_\e(x))\to -m(x)|\nabla u(x)|^{\frac{2-q}{q-1}}\nabla u(x) \hspace{0.4cm} \mbox{for almost every $x\in \Om$}.$$ 
The latter equality and  Egorov's theorem imply that $w= -m|\nabla u|^{\frac{2-q}{q-1}}\nabla u$ from which the second equation in  \eqref{MFGq5} follows. 


On the other hand,  the weak formulation of the first equation in \eqref{MFGqe} reads 
\begin{equation}\label{limit_HJ}
\int_\Om\nabla u_\e\cdot\nabla\varphi\,\dd x +\int_{\Om}H_{q,\e}(-\nabla u_\e)\varphi\,\dd x-\l_\e\int_\Om\varphi\,\dd x -\int_\Om\varphi\,\dd(p_\e+\tilde\a_\e)=\int_\Om f(x,m_\e(x))\varphi(x)\,\dd x,
\end{equation}
for any test function $\varphi\in C^1(\ov\Om)$.  The continuity of $f$ and the dominated convergence theorem imply that 
$$\lim_{\e\to 0}\int_\Om f(x,m_\e(x))\varphi(x)\,\dd x=\int_{\Om}f(x,m(x))\varphi(x)\,\dd x \; \hspace{0.3cm} \mbox{for all $\varphi\in C(\ov\Om)$.}$$
The previous steps imply that we only need to study the limit behavior of the second term in \eqref{limit_HJ}. Since 
$H_{q,\e}(-\nabla u_\e)$ is bounded in $L^{1}(\Om)$, there exists $\gamma \in \fM(\ov{\Om})$ such that, extracting a subsequence, for all $\ds\varphi\in C(\ov{\Om})$,  $\ds\int_{\Om}  H_{q,\e}(-\nabla u_\e)\varphi\, \dd x \to \int_{\Om} \varphi \,\dd \gamma$.  Fatou's lemma implies that 
$$\int_\Om\frac{1}{q'}|\nabla u|^{q'}\varphi\,\dd x\le\liminf_{\e\to 0}\int_\Om H_{q,\e}(-\nabla u_\e)\varphi\,\dd x=\int_\Om\varphi\, \dd\g \; \hspace{0.3cm} \forall \; \; \varphi\in  C(\ov{\Om}), \; \; \varphi \geq 0.$$ 
%
Defining,   $\rho\in\fM_+(\ov\Om)$ as $\dd \rho:=\dd \g-\frac{1}{q'}|\nabla u|^{q'}\dd x$, we obtain that 
\begin{equation}\label{compensation}
\int_\Om\varphi\,\dd\rho+\int_\Om\frac{1}{q'}|\nabla u|^{q'}\varphi\,\dd x=\lim_{\e\to 0}\int_\Om H_{q,\e}(-\nabla u_\e)\varphi\,\dd x \; \; \hspace{0.3cm} \mbox{for all $\varphi\in C(\ov\Om)$.}
\end{equation} 
Thus passing to the limit in \eqref{limit_HJ} as $\e\to 0$ we get  
$$\int_\Om\nabla u\cdot\nabla\varphi\,\dd x +\int_{\Om}\frac{1}{q'}|\nabla u|^{q'}\varphi\,\dd x-\l\int_\Om\varphi\,\dd x -\int_\Om\varphi\,\dd(p+\tilde\a-\rho)=\int_\Om f(x,m(x))\varphi(x)\,\dd x.$$
Setting, $\mu:= \rho-\tilde\a \in \fM_{+}(\ov{\Om})$ we obtain the weak form of the first equation in \eqref{MFGq5}.\smallskip\\
{\it Step 4: Proof of \eqref{weakconcentrationproperty}.}
%
By \eqref{MFGqe} and \eqref{nice_iden} we have
\begin{align*}
0&=\int_\Om(1-m_\e)\,\dd p_\e - \int_\Om m_\e\,\dd\tilde\a_\e=\int_\Om\dd p_\e +\int_\Om m_\e\,\dd(-\tilde\a_\e-p_\e),\\
&=\int_\Om\dd p_\e+\int_\Om\left[\l_\e+f(x,m_\e) - H_{q,\e}(-\nabla u_\e)\right]m_\e\,\dd x-\int_{\Omega} \nabla u_\e\cdot \nabla m_\e\, \dd x,\\
&=\int_\Om\dd p_\e+\int_\Om\left[\frac{1}{q}|v_\e|^q+\frac{\e}{r}|v_\e|^r+\l_\e+f(x,m_\e)\right]m_\e\,\dd x,\\
&\ge \int_\Om\dd p_\e+\int_\Om\left[\frac{1}{q}|v_\e|^q+\l_\e+f(x,m_\e)\right]m_\e\,\dd x.
\end{align*}
By Fatou's lemma we have
$$\int_\Om\frac1q|\nabla u|^{q'}m\,\dd x
\le\liminf_{\e\to 0}\int_\Om\frac1q|v_\e|^qm_\e\,\dd x.$$
Thus, letting $\e\to 0$ and using \eqref{definiciondualidad}, we get
\begin{align*}
0&\ge \int_\Om\dd p+\int_\Om\left[\frac{1}{q}|\nabla u|^{q'}+\l+f(x,m)\right]m\,\dd x\\
&= \int_\Om\dd p+\int_\Om \nabla u \cdot \nabla m\, \dd x +\int_\Om |\nabla u|^{q'}m\, \dd x +\langle \mu-p,m\rangle.
\end{align*}
By taking $u\in W^{1,q'}_\diamond(\Om)$ as test function in the second equation of $(MFG_{q})$ we obtain that
$$\int_\Om\left[\nabla m\cdot\nabla u+m|\nabla u|^\frac{2-q}{q-1}|\nabla u|^2\right]\,\dd x=\int_\Om\left[\nabla m\cdot\nabla u+m|\nabla u|^{q'}\right]\,\dd x=0,$$
from which \eqref{weakconcentrationproperty} follows.
%
%
%
\end{proof}
\begin{remark}
Inequality \eqref{weakconcentrationproperty} is a sort of ``weak concentration property''. In fact, by an approximation argument it is easy to see that we can take $C(\ov{\Om})\cap W^{1,q}(\Omega)$ for the space of  test functions in the first equation of \eqref{MFGq5}. Thus, if $m$ is continuous, we would have that $\ds\langle \mu-p, m\rangle= \int_\Om m \,\dd (\mu-p)$ and so  \eqref{weakconcentrationproperty} would imply that 
$$ \ds\int_\Om m\, \dd\mu=0  \hspace{0.3cm} \mbox{and } \; \;  \int_\Om (1-m)\, \dd p=0, \; \; \mbox{i.e. } \; \mbox{{\rm spt}}(\mu) \subseteq \{m=0\} \; \; \mbox{and }    \;     \mbox{{\rm spt}}(p) \subseteq \{m=1\}, $$
as in Corollary \ref{mfgregularcase}.
\end{remark}
\section*{Appendix}
In this section we recall some classical results about  the regularity of  solutions of elliptic equations with irregular r.h.s. 
%
Recall that we set $\llangle \cdot, \cdot \rrangle$ for the duality product between $ (W^{1,q'}_\diamond(\Omega))^*$ ($q>1$) and $W^{1,q'}_\diamond(\Omega)$.  The following surjectivity result holds true.
\begin{lemma}\label{div_surj}
For any $f\in (W^{1,q'}_\diamond(\Omega))^*$ the weak formulation of
\begin{equation}\label{eqcondivergencia} \mbox{{\rm div}}(F) = f \hspace{0.2cm} \mbox{{\rm in} } \; \Omega, \hspace{0.2cm} 
			       			F \cdot n  = 0  \hspace{0.4cm} \mbox{{\rm in} } \; \partial \Omega, \hspace{0.4cm} \mbox{{\rm i.e.}} \hspace{0.2cm} \; \; -\int_{\Omega} F(x)\cdot \nabla \varphi(x)\, \dd x= \llangle f, \varphi\rrangle \end{equation}
for all $\varphi \in W^{1,q'}_\diamond(\Omega)$,					
has at least one solution $F \in L^{q}(\Omega)^{d}$.
\end{lemma}

\begin{proof}
%
Let us consider the problem 
$$\min_{u\in W^{1,q'}_\diamond}\frac{1}{q'}\int_\Omega|\nabla u|^{q'} \dd x-\llangle f,u \rrangle.$$
Since the cost function is strictly convex, coercive and weakly lower semicontinuous, we have the existence of a unique $u\in  W^{1,q'}_\diamond(\Omega)$ such that  
$$\int_{\Omega} |\nabla u(x)|^{q'-2}\nabla u(x) \cdot \nabla \phi(x) \,\dd x= \llangle f, \varphi\rrangle \; \; \; \forall \; \varphi \in W^{1,q'}_\diamond(\Omega).$$
%
%
The result follows by defining $F=-|\nabla u|^{q'-2}\nabla u \in L^{q}(\Omega)^{d}$. 
%
\end{proof}
Now, given $f\in (W^{1,q'}_\diamond(\Omega))^*$, let us consider the equation 
\begin{equation}\label{edpneumanngenerica}
  -\Delta m = f \hspace{0.4cm} \mbox{in $\Omega$}, \hspace{0.2cm} \nabla m \cdot n= 0 \hspace{0.4cm} \mbox{in $\partial \Omega$}.
\end{equation}
We say that $m\in  W^{1,q}(\Omega)$ is a weak solution of \eqref{edpneumanngenerica} if 
\begin{equation}\label{equationweak}\begin{array}{rcl} \ds  \int_{\Omega} \nabla m(x) \nabla \varphi(x)\, \dd x &=&  \llangle f, \varphi \rrangle \hspace{0.3cm} \forall \; \varphi \in W^{1,q'}_\diamond(\Omega).
\end{array}\end{equation}
\begin{lemma}\label{regularity} Assume that $q>d$ and let  $a\in\mathbb{R}$. Then, there exists a unique weak solution of \eqref{edpneumanngenerica} satisfying  that  $\ds\int_{\Omega} m\, \dd x =a$.  Moreover,  there exists a constant $c>0$, independent of $(a,f)$, such that for any  $F$ solving  \eqref{eqcondivergencia} we have that 
\begin{equation}\label{estimategiacz} \| \nabla m \|_{L^{q}} \leq c \| F\|_{L^{q}}.\end{equation}
\end{lemma}
\begin{proof}[Sketch of the proof:]
Noticing that     \eqref{equationweak} is invariant if a constant is added to   $m$, it suffices to prove the result for $a=0$. 
Since $q>d$ we have that $q'<2$  and so, by the Lax-Milgram theorem,  existence and uniqueness  for   \eqref{equationweak} holds in $W^{1,2}_\diamond(\Omega)$. Using  interpolation results due to  Stampacchia (see \cite{stampacchia1} and \cite{stampacchia2}), estimate \eqref{estimategiacz} holds if Dirichlet-boundary conditions were considered (see e.g. \cite[Theorem 7.1]{giaquinta-martinazzi}). This argument yields the desired local regularity for $m$, which can be extended up to the boundary (which we recall that it is assumed to be regular) using classical reflexion arguments.

\end{proof}
Finally let us recall the following result about elliptic equations with measure data. 
\begin{theorem}[\cite{mingione2}, Theorem 1.2]\label{reg_frac_sob}
Let $f\in\fM(\ov{\Omega}).$ Then the unique solution $u\in W_0^{1,1}(\Om)$ of the problem 
\begin{equation}\label{eq_meas}
-\Delta u = f  \; \; \iin\ \Om, \hspace{0.3cm} u=0  \; \; \oon\ \partial\Om 
\end{equation}
has the following regularity properties:\smallskip\\
{\rm(i)} $\nabla u\in W_{\rm loc}^{1-\e,1}(\Om)^d,$ for all $\e\in(0,1)$.\smallskip\\
{\rm(ii)}  More generally,  $\nabla u\in W_{\rm loc}^{\frac{\s(r)-\e}{r},r}(\Om)^d$, for all $\e\in(0,\s(r))$ where $1\le r<\frac{d}{d-1}$ and $\s(r):=d-r(d-1).$
\end{theorem}

\begin{remark}
We remark that the result about the uniqueness of the (renormalized) solution of the problem \eqref{eq_meas} can be found in \cite{dalmaso}. Moreover, since the regularity results in Theorem \ref{reg_frac_sob} are local, these remain true if we use homogeneous Neumann boundary conditions instead of Dirichlet ones. In this context the solution is unique up to an additive constant. 
\end{remark}

\bibliographystyle{elsarticle-num}
\bibliography{alparsilva}{}

\end{document}